\documentclass[11pt]{article}
\usepackage{amsmath,amssymb,stmaryrd,amsthm,MnSymbol}
\usepackage{bm}
\usepackage[a4paper, top=2.5cm, bottom=2.5cm, left=3cm, right=3cm]{geometry}
\usepackage{graphicx}
\usepackage{xcolor}
\usepackage[colorlinks=true,
            linkcolor=red,
            filecolor=blue,
            citecolor=blue,
            urlcolor=blue,]{hyperref}
\usepackage{cancel}
\usepackage{booktabs}
\usepackage{subcaption}
\usepackage{placeins}
\usepackage{float}

\newcommand{\pp}[2]{\frac{\partial #1}{\partial #2}} 
\newcommand{\dede}[2]{\frac{\delta #1}{\delta #2}}

\DeclareMathOperator{\diff}{d}

\newtheorem{proposition}{Proposition}
\newtheorem{remark}{Remark}
\newtheorem{lemma}{Lemma}

\begin{document}
\title{Energy conserving compatible space--time FEM for deterministic and stochastic rotating shallow water}
\author{Maneesh Kumar Singh\textsuperscript{1,2} and
        Colin J. Cotter\textsuperscript{2}\\[6pt]
  \small \textsuperscript{1}Department of Mathematics, SRM Institute of Science
  and Technology, Kattankulathur, India\\
  \small \textsuperscript{2}Department of Mathematics, Imperial College London,
  London SW7 2AZ, UK}
\date{}
\maketitle

\begin{abstract}
  We present a fully discrete energy conserving discretisation of the rotating shallow
  water equations with arbitrary order in space and time. A space time finite element
  discretisation is used, built from compatible finite element methods in space
  and continuous Petrov-Galerkin schemes in time. Projected auxiliary variables for
  Hamiltonian derivatives are necessary for exact energy conservation. We provide
  an iterative solver approach based on monolithic patch smoothers, which allows
  us to efficiently converge a Newton method sufficiently that energy errors due
  to the solver are very small. We extend this construction to the energy conserving
  {Stochastic Forcing by Lie Transport} formulation. Then we demonstrate the
  convergence, conservation and robustness in numerical tests.
\end{abstract}


\section{Introduction}\label{sec:intro}

Energy conserving spatial discretisations for geophysical fluid dynamics have been 
considered by a number of researchers. Even though atmosphere and ocean
models combine the fundamentally energy conserving inviscid fluid dynamics with
dissipation and source terms, it can be argued that maintaining energy conservation
within the inviscid component of the model ensures a reliable representation of
transfers of energy between different components and scales of kinetic and potential
energy in the model, in the light of important geophysical balances, such as hydrostatic
and geostrophic balance, where they dominate \cite{cullen2007modelling}. Energy
conservation is an important consideration for climate models \cite{jablonowski2011pros},
and has been demonstrated to be important for good representation of backscatter
in two-dimensional turbulence models \cite{thuburn2014cascades,natale2017scale}. Hence,
it is fruitful to understand how energy conservation can be achieved in spatial discretisations of geophysical fluid dynamics.

Interest in energy conserving discretisations for geophysical fluid dynamics
has usually been in connection with conservation of other quantities such
as mass and potential enstrophy. All of these conserved quantities can be explained
by deriving the equations of motion from Hamilton's principle under the assumption
that fluid particles are transported by a time dependent, invertible, 
Lagrangian flow map. Whilst mass conservation is built into this construction by 
assumption, energy conservation arises from symmetry under translation in time,
whilst enstrophy is one of infinitely many conserved quantities arising from
symmetry under relabelling of Lagrangian particles. A general recipe for how to
translate this derivation to Eulerian coordinates using reduction by particle
labelling symmetry was given in \cite{holm1998euler,holm1999euler}. As described there, this results in an equivalent Poisson bracket formulation, with energy conservation arising from antisymmetry of the bracket, and enstrophy and mass conservation arising as
``Casimirs" of the bracket (quantities that always cause the bracket to vanish when inserted). Any energy conserving scheme amounts
to defining a discrete bracket, and mass and enstrophy conservation can only be 
maintained if the discrete bracket has those quantities as Casimirs.

The story of energy conserving discretisations for geophysical fluid dynamics begins with pioneering and influential work on staggered ``C-grid"
finite difference methods by Sadourny and Arakawa. Sadourny considered 
schemes that conserved either energy or enstrophy\footnote{Mass conservation is always straightforward, and does not require a bracket structure to preserve it, so we do not discuss it any further here.}, Arakawa and Lamb
showed how to conserve both \cite{arakawa1981potential}. Since enstrophy cascades
to small scales in two or three dimensions, it is actually undesirable to conserve
it, and so Sadourny recommended to use the energy conserving scheme with an additional
enstrophy dissipating term which also has an interpretation as an eddy parameterisation 
\cite{sadourny1985parameterization}. Arakawa and Hsu also considered schemes that conserve energy and dissipate enstrophy \cite{arakawa1990energy}. Following work by Sadourny, these C-grid schemes were extended to arbitrary cell shapes to achieve
quasi-uniform coverage of the sphere and to allow for local mesh refinement and non-orthogonal grids, whilst considering linear geostrophic balance and consistent potential vorticity transport \cite{thuburn2009numerical,ringler2010unified,thuburn2014mimetic}.
This work was strongly informed by the discrete exterior calculus \cite{hirani,thuburn2012framework}.

The use of Poisson brackets to produce energy conserving equations was proposed by
Morrison \cite{morrison1981hamiltonian,morrison1998}, and in the context of geophysical fluid dynamics by Salmon \cite{salmon2004poisson}, with more recent work by \cite{dubos2014equations,eldred2017total}. A key difficulty is that, with 
the exception of the Zeitlin spectral truncations for incompressible flows \cite{zeitlin1991finite,zeitlin2004self}, the Jacobi identity does not survive 
discretisation. However, the preservation of energy only relies upon antisymmetry,
and it is possible to maintain the Casimir property for total vorticity and potential
enstrophy; indeed the original Arakawa and Lamb scheme can be interpreted this way.
\cite{salmon2005general} proposed to extend this idea to Nambu brackets to obtain
further conserved quantities in more general discretisations; this was followed up by
\cite{gassmann2008towards,gassmann2013global,nevir2009energy,sommer2009conservative}.
The bracket approach was extended to compatible finite elements (discretisations built from the finite element exterior calculus) to 
produce energy-enstrophy conserving schemes for the rotating shallow water equations in \cite{mcrae2014energy}.
An energy conserving/enstrophy dissipating scheme was demonstrated in \cite{bauer2018energy}, and upwinding of density, layer depth and 
temperature were incorporated into energy conserving schemes in \cite{wimmer2020energy,wimmer2021energy}. A related approach based on mimetic spectral
elements has been followed in \cite{lee2018mixed,lee2018discrete,lee2020mixed},
with a general survey in \cite{cotter2023compatible}.

An alternative approach to deriving discrete fluid models was proposed by \cite{pavlov2011structure,gawlik2011geometric}, using the Koopman operator view
of Lagrangian flow maps, which embeds a representation of these maps in the group of general linear maps from a function space to itself. This approach allows to translate the Hamilton's principle derivation of fluid mechanics to the discrete level by 
restricting that function space to a finite dimensional subspace. Unfortunately,
the group structure of the embedding is not preserved, requiring nonholonomic 
constraints to be introduced, which seems to remove the possibility of enstrophy
conservation. However, energy conservation survives to the spatially discrete
level, and even fully discrete variational integrators are possible. Originally
formulated in the framework of discrete exterior calculus, this approach was 
extended to finite element exterior calculus by \cite{natale2018variational,gawlik2021variational}.

To date, there has been relatively little consideration of energy conserving
time integrators to complement these Eulerian energy conserving spatial
discretisations. This is partially because these time integrators are 
implicit, and the scalable solution of the resulting sparse nonlinear
systems has not been addressed. Historically, conservative time integration
methods have focussed mostly on symplectic and variational integrators \cite{hairer1999numerical,leimkuhler2004simulating}. However, the spatial
discretisations discussed above do not have a Jacobi identity, so there is
no symplectic structure to preserve. The Koopman operator approach allow
variational integrators but it seems that there are issues with the modified
Hamiltonian in this function space setting; in \cite{gawlik2021variational}
an energy conserving time discretisation allowed stability for larger timesteps
than the variational time discretisation. Hence, we focus on energy conserving
time discretisations in this work.

Gauss-Legendre implicit Runge-Kutta methods have the property that
they conserve all quadratic invariants of the ODE, so if the energy is
quadratic, these time integrators will conserve it \cite[Theorem
  2.1]{hairer1999numerical}.  When the energy is not quadratic, more
bespoke methods are required. Discrete gradient methods are built from
approximations of the gradient of the Hamiltonian that have a finite
directional derivative property
\cite{gonzalez1996time,mclachlan1999geometric}, with higher order
extensions provided in
\cite{mclaren2004integral,mclaren2003bootstrapping,norton2013projection,eidnes2022order}
and references therein. Another approach to energy conservation is the
averaged vector field method
\cite{quispel2008new,hairer2010energy}. This was extended to energy
conserving systems with state dependent structure matrix (such as the
geophysical fluid dynamics discretisations under consideration here)
in \cite{cohen2011linear}; their lowest order example time averages
the gradient of the Hamiltonian before being multiplied by the
structure matrix evaluated on the average of the state over the two
time levels. Their higher order schemes extend this by projecting the
gradient of the Hamiltonian into a polynomial space before using this
projected gradient in a Gauss-Legendre collocation scheme. Yet another
approach is the continuous Petrov-Galerkin in time (cPG) method, a
time finite element method that uses degree $k+1$ polynomials for the
solution, and degree $k$ polynomials for the test space (hence
Petrov-Galerkin)
\cite{schieweck2010stable,ern2021}. \cite{egger2021energy} observed
that cPG schemes yield discrete gradient methods and averaged vector
field methods after approximating the integrals using numerical
quadrature. The cPG method naturally conserves quadratic invariants of
the continuous time system.  \cite{giesselmann2025energy} showed that
cPG schemes can conserve nonquadratic Hamiltonians using projected
gradients, noting that this recovers the scheme of
\cite{cohen2011linear} after using the Gauss-Legendre quadrature rule
to approximate the time integral in the Galerkin formulation (but the
integrals implementing the Hamiltonian projection must be
exact). \cite{andrew_farrell_sisc25} proposed a general framework for
enforcing conserved nonquadratic quantities at the time discrete level
by introducing auxiliary variables which are the time projection of
the derivatives of the conserved quantities. In the case of energy
conservation, these variables are the projected gradients of the
Hamiltonian of \cite{cohen2011linear}.

To make this cPG scheme with projections (PcPG) concrete, assume that after spatial
discretisation of a partial differential equation (PDE) by finite element methods
using a finite element space $Z$, we have the system for $z\in Z$ such that
\begin{equation}\label{eq:intro-ah}
  \left\langle \mathrm{w}, z_t \right\rangle
  + a_h\!\left(z;\, \mathrm{w},\, \dede{H}{z}\right) = 0,
  \quad \forall \mathrm{w} \in Z,
\end{equation}
where the variational derivative $\dede{H}{z}$ of the Hamiltonian $H$
is defined by
\begin{equation}
\left\langle w,\dede{H}{z}\right\rangle
= \lim_{\epsilon\to 0}\frac{H(z+\epsilon w)-H(z)}{\epsilon}.
\end{equation}
Here, $a_h$ antisymmetric and linear in its last two arguments, but
can depend arbitrarily on $z$ in the first argument. This ensures that
the
energy $H$ satisfies $\dot{H} = -a_h(z; \dede{H}{z}, \dede{H}{z}) = 0$. 
To obtain
a cPG scheme, we introduce polynomial spaces $\mathbb{P}_k$ of degree $k$ on the time 
interval $[t_n,t_{n+1}]$. The standard cPG scheme seeks $z \in Z\times \mathbb{P}_k$
with $z(t_n)=z_n$, such that 
\begin{equation}\label{eq:cPG}
\int_{t_n}^{t_{n+1}}
\left\langle \mathrm{w}, z_t \right\rangle
  + a_h\!\left(z;\, \mathrm{w},\, \dede{H}{z}\right)
  \diff t = 0, \quad \forall \mathrm{w} \in Z\times \mathbb{P}_{k-1},
\end{equation}
before taking $z^{n+1}=z(t_{n+1})$. When $a_h(z;v,w)$ is independent of the
first argument (corresponding to state independent structure matrix), this scheme conserves energy. For the general case (corresponding to state dependent structure
matrix), we must apply a projection following \cite{cohen2011linear}, 
producing the modified system
\begin{equation}\label{eq:PcPG}
\int_{t_n}^{t_{n+1}}
\underbrace{\left\langle \mathrm{w}, z_t \right\rangle}_{\star}
  + a_h\!\left(z;\, \mathrm{w},\, \mathbb{P}\dede{H}{z}\right)
  \diff t = 0, \quad \forall \mathrm{w} \in Z\times \mathbb{P}_{k-1},
\end{equation}
where we have defined the projection $\mathbb{P}\dede{H}{z} \in Z\times \mathbb{P}_{k-1}$ by
\begin{equation}
\label{eq:P}
\int_{t_n}^{t_{n+1}}
\left\langle w, \mathbb{P}\dede{H}{z} - \dede{H}{z}\right\rangle\diff t =0,\quad 
\forall w\in Z\times \mathbb{P}_{k-1},
\end{equation}
i.e. the $L^2$ space-time projection of $\dede{H}{z}$ into $Z\times \mathbb{P}_{k-1}$.
We can then verify energy conservation by direct computation,
\begin{align}
  \label{eq:H discrete}
H^{n+1} - H^n &= \int_{t_n}^{t_{n+1}} \dot{H}(z) \diff t, \\
&= \int_{t_n}^{t_{n+1}} \left\langle \dede{H}{z}, z_t \right\rangle \diff t, \\
\mbox{[use $w=z_t$ in \eqref{eq:P}]} &= \int_{t_n}^{t_{n+1}}
\left\langle \mathbb{P}\dede{H}{z}, z_t \right\rangle\diff t, \\
\mbox{[use $w=z_t$ in \eqref{eq:PcPG}]} &= -\int_{t_n}^{t_{n+1}} a_h\left(z; \mathbb{P}\dede{H}{z},\mathbb{P}\dede{H}{z}
\right)\diff t = 0,
\end{align}
after noting that $z_t\in Z\times \mathbb{P}_{k-1}$ and hence is a substitution for
\eqref{eq:P} and \eqref{eq:PcPG}. Finally, we note that nothing changes
in this calculation if we replace the time integral of $a_h$ in \eqref{eq:PcPG}
with a quadrature rule,
\begin{equation}
\int_{t_n}^{t_{n+1}} a_h(z;v,w) \diff t
\mapsto A_h(z;v,w) = \sum_{g=1}^{n_g} \rho_g a_h(z(t_g); v(t_g),w(t_g)),
\end{equation}
where $\{\rho_g\}_{g=1}^{n_g}$ and $\{t_g\}_{g=1}^{n_g}$ are quadrature
weights and quadrature points respectively,
since $A_h$ is still antisymmetric. When the Gauss-Legendre quadrature rule 
is used, we can use it without error in the time derivative term marked
with a star in \eqref{eq:PcPG}, and we have recovered the energy conserving collocation scheme of \cite{cohen2011linear}.

In this paper, we apply the PcPG scheme to a compatible finite element discretisation
of the rotating shallow water equations of the sphere. Previously, the scheme
of \cite{cohen2011linear} was applied to a similar discretisation in \cite{wimmer2020energy}; \cite{eldred2019quasi} having previously introduced this approach to the thermal shallow water equations using a related spline-based spatial discretisation. However, at that time we did not have access to scalable
iterative solver to solve the resulting nonlinear algebraic systems. We had to 
resort to a fixed point method using a solver for the linearisation of the equations
about a state of rest (such as is used in semi implicit time schemes). We were unable 
to converge the iteration to drive energy conservation towards the boundaries of machine precision. Since then, the work of \cite{Farrell2021,Amin2025} has enabled robust automation of additive Schwarz patch smoothers, which have been applied to implicit timestepping for compatible finite element discretisations in \cite{CotterShipton2023,bauer2025fully}. The automation of these schemes in Firedrake
\cite{farrell2021irksome,andrews2026irksome}.

Recently, we have seen the development of stochastic parameterisations of 
unresolved motions in geophysical fluid models that concern themselves with 
conservation. In the transport under location
uncertainty \cite{memin2014fluid, resseguier2017geophysical,brecht_JAMES}, energy
is preserved in expectation. Another approach, Stochastic
Advection by Lie Transport (SALT) \cite{holm2015variational, cotter2019numerically,holm2026geometry},
provides a stochastic version of the Kelvin circulation theorem.
A third approach, Stochastic Forcing by Lie Transport
(SFLT) \cite{holm2021stochastic, holmsinghsteet2026}, conserves energy by
inserting the stochastic transport into the Poisson bracket. A SFLT stochastic parameterisation was investigated in the context of the 3D primitive equations in \cite{hupatching2023variational}. An energy conserving discretisation of the SFLT extension of the
Camassa--Holm equation in \cite{holmsinghsteet2026}. The Camassa-Holm
equation has a conserved quadratic energy which is also conserved by the SFLT extension of it. The implicit midpoint rule conserves all quadratic invariants, and this was used to produce an energy conserving scheme in \cite{holmsinghsteet2026}, where it was beneficial for long time solutions, to isolate differences between SFLT and SALT in that setting.
Hence, it is interesting to see if energy conserving discretisations can be extended to SFLT models with nonquadratic energy, such as the rotating 
shallow water equations on the sphere. 

The rest of this paper is organised as follows.
Section~\ref{sec:discretemethod} gives the spatial discretisation, establishes
antisymmetry and introduces the PcPG time
discretisation and the quadrature required, states the fully discrete scheme, and
extends it to the stochastic case. Section~\ref{sec:numerics} presents numerical
results for analytic and non-analytic test cases, in both the deterministic and
the stochastic setting. Section~\ref{sec:discussion} discusses the results and ongoing work.

\section{Numerical method}\label{sec:discretemethod}
We consider the rotating shallow water equations in \textit{vector-invariant} form,
\begin{align}
  P_S\left(\pp{\bm{u}}{t} + \bm{u}^\perp(\nabla^{\perp} \cdot \bm{u}+ f) + \nabla\left(\frac{|\bm{u}|^2}{2}
  + g(D+b)\right)\right) & = 0, \label{eq:swe-u}\\
  \pp{D}{t} + \nabla\cdot(\bm{u}D) & = 0, \label{eq:swe-D}
\end{align}
where $\bm{u}$ is the velocity vector field tangential to the sphere, $P_S$ is the orthogonal projection into that tangent space, $D$ the layer depth, $b$ the height at the bottom
of the fluid layer, $f$ the Coriolis parameter, $g$ the acceleration due to
gravity, $\bm{u}^{\perp} = \hat{\bm{k}}\times \bm{u}$, $\hat{\bm{k}}$ the outward
pointing normal to the surface (e.g.\ on the sphere), and $\nabla$ the restriction
of the gradient tangential to the surface.

For smooth solutions, \eqref{eq:swe-u}--\eqref{eq:swe-D} conserve the energy
\begin{equation}\label{eq:H}
  H = \int_\Omega \frac{D|\bm{u}|^2}{2} + gD\left(\frac{D}{2} + b\right) \diff x .
\end{equation}
We describe the energy conserving spatial discretisation in Section \ref{sec:spatial},
and the energy conserving time discretisation in Section \ref{sec:time}.
\subsection{Spatial discretisation}
\label{sec:spatial}
The compatible finite element spatial discretisation follows the upwinded energy conserving scheme of \cite{wimmer2020energy} (following a similar scheme for
the incompressible Euler equations introduced in \cite{natale2018variational})
which avoids the need for an auxiliary potential vorticity variable as was used
in \cite{mcrae2014energy}. However, we do not use the layer depth upwinding of 
\cite{wimmer2020energy}, instead treating the continuity equation as in the same way as \cite{mcrae2014energy}. This is just a convenient choice for us and the
PcPG time scheme could have equally been applied to the schemes presented
in these other works.

We triangulate the sphere by an icosahedral mesh $\mathcal{T}_h$, refined by uniform subdivision, in which each refinement splits every triangle into four. A piecewise polynomial
approximation of the sphere is constructed by moving all of the Lagrange interpolation
points of each triangle radially out to the sphere surface.
On $\mathcal{T}_h$ we take $V\subset H(\mathrm{div})$ and $Q\subset L^2$ to be a compatible pair with $\nabla\cdot V \subseteq Q$, here either the $\mathrm{BDM}_r$--$\mathrm{DG}_{r-1}$ or the $\mathrm{BDFM}_r$--$\mathrm{DG}_{r-1}$ spaces. These $H(\mathrm{div})$ spaces
are constrained to be tangent to the mesh elements by construction, so the projection
$P_S$ is naturally taken care of in the variational formulation.

The variational derivatives $(\dede{H}{\bm{u}},\dede{H}{D})\in V\times Q$
are defined so that
\begin{equation}\label{eq:var-der}
\left\langle \bm{w}, \dede{H}{\bm{u}} - D\bm{u} \right\rangle = 0, \quad \forall \bm{w}\in V,
\quad 
\left\langle \phi, \dede{H}{D} - \left(\frac{|\bm{u}|^2}{2}+g(D+b)\right) \right\rangle = 0, \quad\forall \phi\in Q.
\end{equation}

The vorticity $\nabla^\perp\!\cdot \bm{u}$ is not defined for $\bm{u}\in V$, so we
adopt the upwinded vector-invariant formulation of \cite{natale2018variational},
and seek $(\bm{u},D)\in V\times Q$ such that
\begin{align}
  \langle \bm{w}, \bm{u}_t \rangle
  - \left\langle \nabla^\perp_h\left(\frac{\bm{w}}{D}\cdot{\dede{H}{\bm{u}}}^\perp\right),
  \bm{u} \right\rangle
  + \left\llangle 2\left\{\bm{n}^\perp\left(\frac{\bm{w}}{D}\cdot{\dede{H}{\bm{u}}}^\perp\right)\right\},
  \tilde{\bm{u}} \right\rrangle \notag\\
  \qquad
  + \left\langle \bm{w}/D, {\dede{H}{\bm{u}}}^\perp f \right\rangle
  - \left\langle \nabla\cdot \bm{w}, \dede{H}{D} \right\rangle
  & = 0, \quad \forall \bm{w} \in V,
  \label{eq:weak u}\\
  \left\langle \phi, D_t + \nabla\cdot \dede{H}{\bm{u}} \right\rangle & = 0,
  \quad \forall \phi \in Q,
  \label{eq:weak D}
\end{align}
where $\nabla^\perp_h$ is the element-wise broken derivative,
$\llangle\cdot,\cdot\rrangle$ integrates over all mesh facets, and $\tilde{\bm{u}}$
is the upwind value of $\bm{u}$. The two sides of each interior facet are labelled
$+$ and $-$ with $\bm{n}^- = -\bm{n}^+$, and $\{a\} = \tfrac12(a^+ + a^-)$ denotes
the facet average.

Writing $z=(\bm{u},D)$, $Z = V\times Q$, 
we observe that the system \eqref{eq:weak u}--\eqref{eq:weak D} takes the form of
\eqref{eq:ah-def},
where, for $\mathrm{w}=(\bm{w},\phi)$ and $\mathrm{v}=(\bm{v},\psi)$ in $Z$,
\begin{multline}\label{eq:ah-def}
  a_h(z;\,\mathrm{w},\mathrm{v})
  := - \left\langle \nabla^\perp_h\left(\frac{\bm{w}}{D}\cdot \bm{v}^\perp\right),
     \bm{u} \right\rangle
   + \left\llangle 2\left\{\bm{n}^\perp\left(\frac{\bm{w}}{D}\cdot \bm{v}^\perp\right)\right\},
     \tilde{\bm{u}} \right\rrangle \\
   + \left\langle f,\, \frac{\bm{w}}{D}\cdot \bm{v}^\perp \right\rangle
   - \left\langle \nabla\cdot \bm{w},\, \psi \right\rangle
   + \left\langle \phi,\, \nabla\cdot \bm{v} \right\rangle.
\end{multline}
The following lemma verifies that $a_h$ is antisymmetric.
\begin{lemma}[Antisymmetry of $a_h$]\label{lem:antisym}
For all $z$ and all $\mathrm{w},\mathrm{v}\in Z$,
\begin{equation}\label{eq:ah-antisym}
  a_h(z;\,\mathrm{w},\mathrm{v}) = -\,a_h(z;\,\mathrm{v},\mathrm{w}).
\end{equation}
\end{lemma}
\begin{proof}
Write $\mathrm{w}=(\bm{w},\phi)$, $\mathrm{v}=(\bm{v},\psi)$, and denote by
$s(\mathrm{w},\mathrm{v}) = (\bm{w}/D)\cdot \bm{v}^\perp$ the scalar through which
the advection, facet and Coriolis terms of \eqref{eq:ah-def} depend on
$(\mathrm{w},\mathrm{v})$. Each of these three terms is linear in
$s(\mathrm{w},\mathrm{v})$, with $z$ fixed. From
$\bm{v}^\perp = \hat{\bm{k}}\times \bm{v}$ and the scalar triple product,
\begin{equation}\label{eq:s-antisym}
  \bm{w}\cdot \bm{v}^\perp
  = \hat{\bm{k}}\cdot(\bm{v}\times \bm{w})
  = -\,\hat{\bm{k}}\cdot(\bm{w}\times \bm{v})
  = -\,\bm{v}\cdot \bm{w}^\perp ,
\end{equation}
so $s(\mathrm{v},\mathrm{w}) = -\,s(\mathrm{w},\mathrm{v})$ pointwise, on each
side of every facet. By linearity, the advection, facet and Coriolis terms each
change sign under the interchange. The two divergence terms satisfy
\[
  -\langle \nabla\cdot \bm{v}, \phi\rangle + \langle \psi, \nabla\cdot \bm{w}\rangle
  = -\left( -\langle \nabla\cdot \bm{w}, \psi\rangle + \langle \phi, \nabla\cdot \bm{v}\rangle \right).
\]
Summing the terms of \eqref{eq:ah-def} gives
$a_h(z;\mathrm{v},\mathrm{w}) = -\,a_h(z;\mathrm{w},\mathrm{v})$. 
\end{proof}
\begin{remark}
We note that $a_h$ remains antisymmetric if quadrature rules are used to approximate
these spatial integrals, provided that the same quadrature rule is used to approximate both of the
divergence terms, since the other terms have antisymmetric integrands. This is 
important because the upwinded term cannot be integrated exactly using static quadrature
rules, which is how integration is performed in most finite element codes, including
Firedrake.
\end{remark}

Since $\bm{w}\in V$ implies $\nabla\cdot \bm{w}\in Q$, we may choose $\phi=\nabla\cdot\bm{w}$ in \eqref{eq:var-der}, to obtain
\[
  \left\langle \nabla\cdot \bm{w},\, \dede{H}{D}\right\rangle
  = \left\langle \nabla\cdot \bm{w},\, \frac{|\bm{u}|^2}{2}+g(D+b)\right\rangle,
  \quad \forall \bm{w}\in V,
\]
so that $\dede{H}{D}$ is eliminated. The system reduces to seeking
$(\bm{u},D,\dede{H}{\bm{u}})\in V\times Q\times V$ such that
\begin{align}
  \langle \bm{w}, \bm{u}_t \rangle
  - \left\langle \nabla^\perp_h\left(\frac{\bm{w}}{D}\cdot \dede{H}{\bm{u}}^\perp\right),
  \bm{u} \right\rangle
  + \left\llangle 2\left\{\bm{n}^\perp\left(\frac{\bm{w}}{D}\cdot \dede{H}{\bm{u}}^\perp\right)\right\},
  \tilde{\bm{u}} \right\rrangle \notag\\
  \qquad
  + \left\langle \bm{w}/D, \dede{H}{\bm{u}}^\perp f \right\rangle
  - \left\langle \nabla\cdot \bm{w}, \frac{|\bm{u}|^2}{2}+g(D+b) \right\rangle
  & = 0, \quad \forall \bm{w} \in V,
  \label{eq:semi u}\\
  \left\langle \phi, D_t + \nabla\cdot \dede{H}{\bm{u}} \right\rangle & = 0,
  \quad \forall \phi \in Q,
  \label{eq:semi D}\\
  \left\langle \dede{H}{\bm{u}} - D\bm{u}, \bm{\gamma} \right\rangle & = 0,
  \quad \forall \bm{\gamma} \in V.
  \label{eq:semi F}
\end{align}

\subsection{Time discretisation}
\label{sec:time}
Now we apply the PcPG time discretisation to (\ref{eq:weak u}-\ref{eq:weak D}). 
We seek $(\bm{u},D)$ in the affine trial spaces
\begin{equation}\label{eq:trial-spaces}
  (V\otimes\mathbb{P}_k)_n = \{\bm{v}\in V\otimes\mathbb{P}_k :
    \bm{v}(\cdot,t^n)=\bm{u}^n\},
  \qquad
  (Q\otimes\mathbb{P}_k)_n = \{\psi\in Q\otimes\mathbb{P}_k :
    \psi(\cdot,t^n)=D^n\},
\end{equation}
and take test functions $(\bm{w},\phi)\in
(V\otimes\mathbb{P}_{k-1})\times(Q\otimes\mathbb{P}_{k-1})$. For $\bm{v}\in L^2(I^n;V)$ and
$\psi\in L^2(I^n;Q)$, the space--time $L^2$ projections
$\Pi_V\bm{v}\in V\otimes\mathbb{P}_{k-1}$ and $\Pi_Q\psi\in Q\otimes\mathbb{P}_{k-1}$
are defined by
\begin{equation}\label{eq:proj}
\begin{aligned}
  \int_0^{\Delta t}\!\left\langle \Pi_V\bm{v} - \bm{v},\, \bm{\gamma} \right\rangle \diff t &= 0
  &&\forall\,\bm{\gamma}\in V\otimes\mathbb{P}_{k-1}, \\
  \int_0^{\Delta t}\!\left\langle \Pi_Q\psi - \psi,\, \chi \right\rangle \diff t &= 0
  &&\forall\,\chi\in Q\otimes\mathbb{P}_{k-1}.
\end{aligned}
\end{equation}
Using this notation, the PcPG discretisation of (\ref{eq:weak u}-\ref{eq:weak D})
takes the form
\begin{align}
  \int_{t_n}^{t_{n+1}}\langle \bm{w}, \bm{u}_t \rangle
  - \left\langle \nabla^\perp_h\left(\frac{\bm{w}}{D}\cdot{\Pi_V\dede{H}{\bm{u}}}^\perp\right),
  \bm{u} \right\rangle
  + \left\llangle 2\left\{\bm{n}^\perp\left(\frac{\bm{w}}{D}\cdot{\Pi_V\dede{H}{\bm{u}}}^\perp\right)\right\},
  \tilde{\bm{u}} \right\rrangle \notag\\
  \qquad
  + \left\langle \bm{w}/D, {\Pi_V\dede{H}{\bm{u}}}^\perp f \right\rangle
  - \left\langle \nabla\cdot \bm{w}, \Pi_Q\dede{H}{D} \right\rangle\diff t
  & = 0, \quad \forall \bm{w} \in V,
  \\
  \int_{t_n}^{t_{n+1}}\left\langle \phi, D_t + \nabla\cdot \Pi_V\dede{H}{\bm{u}} \right\rangle\diff t & = 0,
  \quad \forall \phi \in Q, \\  
\int_{t_n}^{t_{n+1}}\left\langle \bm{\gamma}, \Pi_V\dede{H}{\bm{u}} - D\bm{u} \right\rangle
\diff t&= 0, \quad \forall \bm{\gamma}\in V,\\
\int_{t_n}^{t_{n+1}}\left\langle \chi, \Pi_Q\dede{H}{D} - \left(\frac{|\bm{u}|^2}{2}+g(D+b)\right) \right\rangle\diff t &= 0, \quad\forall \chi\in Q.
\end{align}
We repeat our previous remarks about quadrature rules: energy conservation
still holds if we use incomplete quadrature in space and time for the terms
corresponding to $a_h$, but the terms involving time derivatives and 
Hamiltonian derivatives must be integrated exactly. Since
$\dede{H}{\bm{u}}$ is the projection of $D\bm{u}$, with $\bm{u},D\in\mathbb{P}_k$
the product $D\bm{u}$ has degree $2k$ in time, and the defining integral
\eqref{eq:proj} pairs it with $\bm{\gamma}\in\mathbb{P}_{k-1}$, giving degree
$2k+(k-1)=3k-1$. A $n_q$-point Gauss--Legendre rule integrates degree at most  $ 2n_q -1$
exactly, so the projection requires
\begin{equation}\label{eq:quad-req}
  n_q \;\ge\; \Bigl\lceil \tfrac{3k}{2} \Bigr\rceil,
\end{equation}
strictly more than the $n_q \ge k$ needed for the Galerkin integrals
$\langle \bm{w},z_t\rangle$ of degree $2k-2$.

We may eliminate the variable $\Pi_Q\dede{H}{D}$ in a similar manner
to the spatial semidiscretisation. Further, the PcPG scheme carries the
trial variables $(\bm{u},D)$ at degree $k$ but the projected derivative
$\Pi_V\dede{H}{\bm{u}}$ at degree $k-1$. To place all trial functions in the same
time space we introduce a variable $\bm{G}\in V\otimes\mathbb{P}_k$ with
$\bm{G}_t=\Pi_V\dede{H}{\bm{u}}$, enforced variationally,
\begin{equation}\label{eq:G-eq}
  \int_0^{\Delta t}\left\langle \bm{\gamma},\, \bm{G}_t - \dede{H}{\bm{u}} \right\rangle\diff t = 0,
  \quad \forall\, \bm{\gamma}\in V\otimes\mathbb{P}_{k-1}.
\end{equation}

Collecting the spatial discretisation, the PcPG($k$) time discretisation and
\eqref{eq:G-eq}, the fully discrete system seeks
$(\bm{u},D,\bm{G})\in(V\otimes\mathbb{P}_k)\times(Q\otimes\mathbb{P}_k)\times(V\otimes\mathbb{P}_k)$
with $\bm{u}(\bm{x},0)=\bm{u}^n$, $D(\bm{x},0)=D^n$ and $\bm{G}(\bm{x},0)=\bm{0}$,
such that
\begin{align}
\label{eq:discrete u 2}
&\int_0^{\Delta t}\left\langle \bm{w}, \bm{u}_t \right\rangle
 - \left\langle \nabla_h^\perp\!\left(\frac{\bm{w}}{D}\cdot \bm{G}_t^\perp\right), \bm{u} \right\rangle
 + \left\llangle 2\left\{\bm{n}^\perp\!\left(\frac{\bm{w}}{D}\cdot \bm{G}_t^\perp\right)\right\}, \tilde{\bm{u}} \right\rrangle \notag\\
&\qquad
 + \left\langle \bm{w}/D, \bm{G}_t^\perp f \right\rangle
 - \left\langle \nabla\cdot \bm{w},\, \frac{|\bm{u}|^2}{2}+g(D+b) \right\rangle
 \,\diff t = 0,
 \quad \forall\, \bm{w} \in V\otimes\mathbb{P}_{k-1}, \\[4pt]
&\int_0^{\Delta t}\left\langle \phi,\, D_t + \nabla\cdot \bm{G}_t \right\rangle
 \,\diff t = 0,
 \quad \forall\, \phi \in Q\otimes\mathbb{P}_{k-1}, \label{eq:discrete-D-full} \\[4pt]
&\int_0^{\Delta t}\left\langle \bm{G}_t - D\bm{u},\, \bm{\gamma} \right\rangle
 \,\diff t = 0,
 \quad \forall\, \bm{\gamma} \in V\otimes\mathbb{P}_{k-1},
 \label{eq:discrete D 2}
\end{align}
which is the form that we implemented for our numerical experiments (after 
quadrature choices).

\subsection{Extension to the stochastic rotating shallow water equations with SFLT noise}
The SFLT modelling framework produces stochastic partial differential
equations (SPDEs) that conserve energy. We consider Equation (4.19)
of \cite{holm2021stochastic} with vector field stochastic fluctuations
only, so that their scalar fluctations $g_k$ vanish. Writing this in
our present notation and units, we get
\begin{align}\nonumber
{\diff}\bm{u} + (\nabla^\perp \cdot\bm{u} + f)\bm{u}^\perp
{\diff}t + \nabla \left(\frac{1}{2}|\bm{u}|^2 + g(D+b)\right)\diff t
\quad & \\
\qquad + {\frac{1}{D}}\left({\bm{u}}^{\perp}\nabla^\perp\cdot\diff \bm{B} + {\diff \bm{B} \nabla\cdot \bm{u}
+ \nabla \left(\bm{u}\cdot \diff \bm{B}\right)}\right)
& = 0,
\label{eq:stochastic u}
\\
\label{eq:stochastic D}
\diff D + \nabla\cdot(\bm{u}D)\diff t & = 0,
\end{align}
where $\bm{B}$ is a spacetime noise that is Stratonovich in time.

Using the same Hamiltonian as the deterministic case, these equations
have the following equivalent Hamiltonian form,
\begin{align}\nonumber
{\diff}\bm{u} + (\nabla^\perp \cdot\bm{u} + f)\frac{1}{D}\dede{H}{\bm{u}}^\perp{\diff}t+ \nabla \dede{H}{D} \diff t
\quad & \\
\qquad + {\frac{1}{D}}\left(\frac{1}{D}\dede{H}{\bm{u}}^{\perp}\nabla^\perp\cdot\diff \bm{B} 
{+ \diff \bm{B} \nabla\cdot \left(\frac{1}{D}\dede{H}{\bm{u}}\right)
+ \nabla \left(\frac{1}{D}\dede{H}{\bm{u}}\cdot \diff \bm{B}\right)}\right)  & = 0, \\
\diff D + \nabla\cdot\dede{H}{\bm{u}}\diff t & = 0.
\end{align}
We can then formally verify pathwise energy conservation since
\begin{align}
\diff H =& \left\langle \dede{H}{\bm{u}}, \diff \bm{u} \right\rangle +
\left\langle \dede{H}{D}, \diff D \right\rangle, \\
=& -\left\langle \dede{H}{\bm{u}}, \dede{H}{\bm{u}}^\perp(\nabla^\perp \cdot\bm{u} + f)\frac{1}{D} \right\rangle\diff t
- \left\langle \dede{H}{\bm{u}}, \nabla \dede{H}{D} \right\rangle\diff t \nonumber \\
& -\left\langle {\frac{1}{D}}\dede{H}{\bm{u}}, \frac{1}{D}\dede{H}{\bm{u}}^{\perp}\nabla^\perp\cdot\diff \bm{B} \right\rangle 
{-
\left\langle \frac{1}{D}\dede{H}{\bm{u}}, \diff \bm{B} \nabla\cdot \left(\frac{1}{D}\dede{H}{\bm{u}}\right)\right\rangle
- \left\langle \frac{1}{D}\dede{H}{\bm{u}}, \nabla \left(\frac{1}{D}\dede{H}{\bm{u}}\cdot \diff \bm{B}\right)\right\rangle }
\nonumber \\
& - \left\langle \dede{H}{D}, \nabla\cdot \dede{H}{\bm{u}}\right\rangle \diff t = 0,
\end{align}
after integrating gradients by parts and using the boundary conditions $\bm{u}\cdot\bm{n}=0$, making use
of the Stratonovich chain rule. We also note that mass conservation 
is unaffected since the continuity equation remains in flux form.

To obtain a fully discrete scheme, we use discrete time increments
\begin{equation}
  \delta \bm{B}^n = \int_{t_n}^{t_{n+1}} \diff \bm{B}(t).
\end{equation}
Our choice of spatial discretisation of $\bm{B}$ is described in Appendix
\ref{A:B}. The time discretisation is obtained by replacing
\begin{equation}
  \int_{t_n}^{t_{n+1}} \mathcal{N}(u,D) \diff \bm{B}(t)
\end{equation}
by
\begin{equation}
  \int_{t_n}^{t_{n+1}} \mathcal{N}(u,D)\diff t\frac{\bm{B}}{\Delta t}.
\end{equation}
We expect first order strong convergence, since the cPG scheme is
equivalent to the implicit midpoint rule with $\delta \bm{B}$ frozen
over the timestep, a strong first order scheme \cite{kloeden92}.
Hence, for the stochastic scheme, we only consider degree 1
polynomials in time in our numerical experiments.

We now have all the ingredients to describe the fully discrete scheme. The
curl-noise term involving $\nabla^\perp\cdot \diff\bm{B}$ is discretised using the same
approach as the deterministic approach to $\nabla^\perp\cdot \bm{u}$, whilst
the remaining terms are discretised using a discontinuous Galerkin approach using
a centred flux.


Hence, we seek $(\bm{u}, D,
\bm{G}) \in (V\otimes \mathbb{P}_k)\times (Q\otimes
\mathbb{P}_k)\times (V\otimes \mathbb{P}_k)$ with
$\bm{u}(\bm{x},0)=\bm{u}^n$, $D(\bm{x},0)=D^n$ and
$\bm{G}(\bm{x},0)=0$, such that
\begin{align}
\label{eq:sflt-discrete-u}
&\int_0^{\Delta t}\left\langle \bm{w}, \bm{u}_t \right\rangle
 - \left\langle \nabla_h^\perp\!\left(\frac{\bm{w}}{D}\cdot \bm{G}_t^\perp\right), \bm{u} \right\rangle
 + \left\llangle 2\left\{\bm{n}^\perp\!\left(\frac{\bm{w}}{D}\cdot \bm{G}_t^\perp\right)\right\}, \tilde{\bm{u}} \right\rrangle \notag\\
&\qquad
 - \left\langle \nabla_h^\perp\!\left(\frac{\bm{w}}{D}\cdot \bm{G}_t^\perp\right), \frac{\delta\bm{B}^n}{\Delta t D} \right\rangle
 + \left\llangle 2\left\{\bm{n}^\perp\!\left(\frac{\bm{w}}{D}\cdot \bm{G}_t^\perp\right)\right\}, \frac{\tilde{\delta\bm{B}^n}}{D} \right\rrangle \notag\\
&\qquad
 + \left\langle \bm{w}/D, \bm{G}_t^\perp f \right\rangle
 - \left\langle \nabla\cdot \bm{w},\, \frac{|\bm{u}|^2}{2}+g(D+b) \right\rangle
 \notag \\
 &\qquad  \nonumber
{ - \left\langle \nabla_h\left(\frac{\bm{w}\cdot\delta\bm{B}^n}{\Delta t D}\right), \frac{\bm{G}_t}{D}  \right\rangle
 + \left\langle \nabla_h\left(\frac{\bm{G}_t \cdot\delta\bm{B}^n}{\Delta t D}\right), \frac{\bm{w}}{D} \right\rangle }\\
 &\qquad
{ + \left\llangle 2\left\{\bm{n}\left(\frac{\bm{w}\cdot\delta\bm{B}^n}{\Delta t D}\right)\right\}, \left\{\frac{\bm{G}_t}{D}\right\} \right\rrangle - \left\llangle 2\left\{\bm{n}\left(\frac{\bm{G}_t\cdot\delta\bm{B}^n}{\Delta t D}\right)\right\}, \left\{\frac{{\bm{w}}}{D}\right\} \right\rrangle }
 \,\mathrm{d}t 
 = 0,
 \quad \forall \bm{w} \in V\otimes\mathbb{P}_{k-1}, \\[4pt]
&\int_0^{\Delta t}\left\langle \phi,\, D_t + \nabla\cdot \bm{G}_t \right\rangle
 \,\mathrm{d}t = 0,
 \quad \forall \phi \in Q\otimes\mathbb{P}_{k-1}, \\[4pt]
&\int_0^{\Delta t}\left\langle \bm{G}_t - D\bm{u},\, \bm{\gamma} \right\rangle
 \,\mathrm{d}t = 0,
 \quad \forall \bm{\gamma} \in V\otimes\mathbb{P}_{k-1},
 \label{eq:sflt-discrete-G}
\end{align}
where $\tilde{\delta\bm{B}^n}$ is the upwind value of $\delta\bm{B}^n$, taken with the
facet orientation $\operatorname{sign}(\bm{u}\cdot \bm{n})$ used for
$\tilde{\bm{u}}$.

It is then straightforward to demonstrate energy conservation.
\begin{proposition}\label{prop:sflt-energy}
The fully discrete stochastic scheme
\eqref{eq:sflt-discrete-u}--\eqref{eq:sflt-discrete-G} conserves the energy
pathwise, $H(z^{n+1})=H(z^n)$ for every realisation.
\end{proposition}

\begin{proof}
  Take $\bm{w}=\bm{G}_t$. The deterministic terms sum to zero
  following the computations of \eqref{eq:H discrete}. The
  curl-noise terms carry the pointwise factor
  $(\bm{G}_t/D)\cdot\bm{G}_t^\perp=0$. In the remaining terms, setting
  $\bm{w}=\bm{G}_t$ makes the two volume integrals coincide up to
  sign, $-\left\langle \nabla_h\psi,\, \bm{G}_t/D\right\rangle
  +\left\langle \nabla_h\psi,\, \bm{G}_t/D\right\rangle=0$ with
  $\psi=(\bm{G}_t\cdot\delta\bm{B}^n)/D$, and likewise the two facet
  integrals cancel. Hence $H(z^{n+1})=H(z^n)$ pathwise.
\end{proof}

\subsection{Iterative solution strategy}

Finally, we discuss our strategy for solving the resulting sparse
nonlinear systems for the basis coefficients.  In our previous work
\cite{wimmer2020energy}, a fixed point iteration based on the
linearisation about a state of rest was used.  A hybridisable mixed
formulation for the linear system was used, which was condensed down
to a single ``trace" variable before solution using a sparse direct
solver (or multigrid, for larger systems). However, we struggled to
converge this iteration in order to drive the energy conservation
error below a certain level.  As an alternative, here we use a
Newton-Krylov method, and hence need to solve the full Jacobian
linearised about the current state, rather than the state of
rest. Having initially investigated using the linearisation about a
state of rest to precondition the full Jacobian without success, we
have instead chosen to use GMRES applied to the monolithic system for
corrections to $(\bm{u},D,\bm{G})$, preconditioned by an additive
Schwarz method using overlapping vertex star patches. In each patch
consisting of the set of cells surrounding one vertex, the normal
component of the corrections to $\bm{u}$ and $\bm{G}$ are held fixed,
and the coupled system for all three variables is solved using a
direct solver. Then the contributions from all of the overlapping
patches are combined using a partition-of-unity weighted summation
which forms the result of the preconditioner. This is implemented
using Firedrake's PCPATCH implementation. To be precise, each
preconditioner application consists of 3 Richardson iterations with
damping factor 0.8, themselves preconditioned by the additive Schwarz
method. This addresses a balance between optimised Krylov method
convergence and the size of the required subspace. Here, we do not use
a multigrid method, since $\Delta t$ is kept in proportion to $\Delta
x$ as the mesh is refined, maintaining a constant Courant number
(unless specified in the test case description). In
\cite{bauer2025fully} we showed that a multigrid method can extend the
region of Courant number robustness and accelerate time to solution
beyond that of an IMEX method. However, we avoid those complications
here.

\section{Numerical examples}\label{sec:numerics}
In this section, we present numerical examples by means of various
test cases, implemented using Firedrake. The goal of these tests is to
demonstrate the stability, accuracy and robustness of the
discretisation, in the deterministic and the stochastic formulations.
The first test case considers the unsteady analytic solution of
\cite{lauter2005unsteady} (Test case~1), for which the discretisation
error is measured against the exact fields, and Williamson test cases
6 and 5 \cite{williamson1992standard} (Test cases~2 and~3), for which
it is measured against fine-grid numerical references. For Test
cases~1 and~2 we report spatial and temporal convergence rates. For
all test cases we report the relative energy drift; any change in
energy is due to truncation of the iterative implicit solver, meaning
that we only have an approximate implementation of the energy
conserving scheme. To examine solver robustness, the average number of
GMRES iterations per step under mesh refinement is reported.

All computations use an icosahedral triangulation of the sphere, refined by
uniform subdivision. Each refinement level splits every triangle into four, so
the cell count quadruples from one level to the next: refinement levels $2$
through $5$ have $320$, $1{,}280$, $5{,}120$ and $20{,}480$ cells respectively,
with degree-of-freedom counts scaling by the same factor of four per level. For example, at
the finest level~$5$ the BDM$_2$/DG$_1$ pair has $153{,}600$ velocity and
$61{,}440$ depth degrees of freedom. The degree of the finite element spaces
is reported in the examples, 
along with the degree of the polynomial approximation of the sphere surface, which
we called the ``mesh degree".
We used a spatial quadrature of degree 6
throughout (i.e., polynomials of degree 6 would be integrated exactly using the quadrature).
A time quadrature scheme of degree $2s$ for time schemes with $s$ stages is used.

We assess accuracy through the $L^2$ error of a field
$\phi \in \{\bm{u}, D\}$ at the final time, computed against a reference
$\phi_\star$. For the Test case~1, the exact solution is known, so we can denote
$\phi_\star = \phi_{\mathrm{ex}}$.  The relative $L^2$ error is given by
\begin{equation}
  E_{\mathrm{rel}}(\phi;\ell)
    = \frac{\|\phi_\ell - \phi_{\mathrm{ex}}\|_{L^2}}
           {\|\phi_{\mathrm{ex}}\|_{L^2}} ,
  \label{eq:rel-err}
\end{equation}
where $\phi_\ell$ is the computed field at refinement level $\ell$. For Test cases~2 and~3 no exact solution is available. Instead, the reference
$\phi_\star = \phi_{\mathrm{ref}}$ is a fine-grid numerical solution of the same
study, obtained at the finest resolution stated in each table caption. To compute errors,
fields from lower resolutions are interpolated to the finer resolution before the difference is taken.
Since the reference is itself approximate, no natural denominator is defined.
These tables therefore report the unnormalised $L^2$ error
\begin{equation}
  E_{\mathrm{abs}}(\phi;\ell)
    = \|\phi_\ell - \phi_{\mathrm{ref}}\|_{L^2} .
  \label{eq:abs-err}
\end{equation}

When verifying spatial convergence, the level $\ell$ is the refinement level with the
time step fixed.  When verifying temporal convergence, $\ell$ is the number of steps with
a fixed mesh refinement level. We index the levels
$\ell_1 < \ell_2 < \cdots < \ell_N$ from coarsest to finest, so that
$\ell_{i-1}$ is the next-coarser level below $\ell_i$. Consecutive levels
differ by a factor of two in resolution: each spatial refinement halves the
mesh spacing, and each temporal refinement halves $\Delta t$. Writing
$E_\star \in \{E_{\mathrm{rel}}, E_{\mathrm{abs}}\}$ for whichever error is
tabulated in the study at hand, the observed order of convergence between
the consecutive pair $(\ell_{i-1}, \ell_i)$,
\begin{equation}
  \mathrm{rate}_\phi
    = \log_2\!\frac{E_\star(\phi;\ell_{i-1})}{E_\star(\phi;\ell_i)} ,
  \qquad \phi \in \{\bm{u}, D\},\quad i = 2,\dots,N,
  \label{eq:rate}
\end{equation}
is invariant under normalisation by a constant and is therefore reported
identically in both conventions.

For the stochastic experiments, we measure the strong (pathwise)
error.  In convergence experiments, all resolutions share one Brownian
path, obtained by refining the noise increments in space and time, and
$E$ is averaged over independent realisations before the rate is
formed.

\subsection{Test case with analytic solution}\label{sec:analytic}
The Läuter et al.\ \cite{lauter2005unsteady} solution has the shallow layer moving in solid body rotation
about an axis that precesses around the polar axis at the rotation
rate~$\Omega$. It is one of the few unsteady exact solutions available for these
equations on the sphere, so the discretisation error is measured directly. With
sphere radius~$a$, the rotation axis is
\begin{equation}
  \bm{c}(t) = -\sin\alpha\,\cos(\Omega t)\,\bm{e}_1
                  + \sin\alpha\,\sin(\Omega t)\,\bm{e}_2
                  + \cos\alpha\,\bm{e}_3 ,
  \label{eq:lauter-axis}
\end{equation}
and the exact velocity, layer depth and (centrifugal) orography fields are
\begin{align}
  \bm{u}(\bm{x},t) &= \frac{u_0}{a}\,\bm{c}(t)\times\bm{x}, \label{eq:lauter-u}\\
  g\,D(\bm{x},t)       &= k_1 - k_2
       - \tfrac12\!\left(u_0\,\bm{c}(t)\cdot\hat{\bm{x}} + \Omega z\right)^{2},
       \label{eq:lauter-D}\\
  g\,b(\bm{x})         &= \tfrac12\,\Omega^{2} z^{2} + k_2 ,
       \label{eq:lauter-b}
\end{align}
where $\bm{x}=(x,y,z)$ are Cartesian coordinates, $\hat{\bm{x}}=\bm{x}/a$
is the outward unit normal and $z=\bm{x}\cdot\bm{e}_3$. The parameters are
$a = 6{,}371{,}220~\mathrm{m}$, $\alpha = \pi/4$,
$u_0 = 2\pi a/(12~\mathrm{days}) \approx 38.6~\mathrm{m\,s^{-1}}$,
$\Omega = 7.292\times10^{-5}~\mathrm{s^{-1}}$, $g = 9.8~\mathrm{m\,s^{-2}}$,
$k_1 = 133{,}681~\mathrm{m^2\,s^{-2}}$ and $k_2 = 0$, integrated to
$t = 43{,}200~\mathrm{s}$ ($12~\mathrm{h}$) for the spatial studies (the final
times for the temporal studies are given in the table captions).

Tables \ref{tab:lauter-spatial-coords2} and
\ref{tab:lauter-spatial-bdfm} present relative $L^2$ errors against
the exact solution, for BDM$_2$ and BDFM$_2$ velocity with DG$_1$
depth, degree $1$ in time and $\Delta t$ fixed, so the rates are
spatial. Both pairs exhibit second order spatial convergence in
$\bm{u}$ and $D$.  Table~\ref{tab:lauter-temporal-coords2}
demonstrates second order temporal convergence on a fixed mesh
resolution.

\begin{table}
  \centering
  \begin{tabular}{c c c c c c}
    \toprule
    Mesh refinement & $h$ (km) & $\|\bm{u}-\bm{u}_{\mathrm{ex}}\|/\|\bm{u}_{\mathrm{ex}}\|$ & rate$_{\bm{u}}$
        & $\|D-D_{\mathrm{ex}}\|/\|D_{\mathrm{ex}}\|$ & rate$_D$ \\
    \midrule
    2 & 1785.5 & $1.60\times10^{-2}$ & --   & $4.11\times10^{-3}$ & --   \\
    3 &  892.8 & $2.38\times10^{-3}$ & 2.75 & $1.04\times10^{-3}$ & 1.99 \\
    4 &  446.4 & $4.93\times10^{-4}$ & 2.27 & $2.60\times10^{-4}$ & 2.00 \\
    5 &  223.2 & $1.14\times10^{-4}$ & 2.11 & $6.51\times10^{-5}$ & 2.00 \\
    \bottomrule
  \end{tabular}
  \caption{Spatial convergence for the Lauter et al test case. BDM$_2$/DG$_1$,
    mesh degree $2$, degree $1$ in time, $\Delta t = 21.6$\,s, $t = 12$\,h.}
  \label{tab:lauter-spatial-coords2}
\end{table}

\begin{table}
  \centering
  \begin{tabular}{c c c c c c}
    \toprule
    mesh refinement& $h$ (km) & $\|\bm{u}-\bm{u}_{\mathrm{ex}}\|/\|\bm{u}_{\mathrm{ex}}\|$ & rate$_{\bm{u}}$
        & $\|D-D_{\mathrm{ex}}\|/\|D_{\mathrm{ex}}\|$ & rate$_D$ \\
    \midrule
    2 & 1785.5 & $8.42\times10^{-3}$ & --   & $4.11\times10^{-3}$ & --   \\
    3 &  892.8 & $2.12\times10^{-3}$ & 1.99 & $1.04\times10^{-3}$ & 1.99 \\
    4 &  446.4 & $5.37\times10^{-4}$ & 1.98 & $2.60\times10^{-4}$ & 2.00 \\
    5 &  223.2 & $1.35\times10^{-4}$ & 2.00 & $6.51\times10^{-5}$ & 2.00 \\
    \bottomrule
  \end{tabular}
  \caption{Spatial convergence for the L\"auter et al test case. Same settings as Table \ref{tab:lauter-spatial-coords2}, except that BDFM$_2$/DG$_1$ spaces were used.}
  \label{tab:lauter-spatial-bdfm}
\end{table}


\begin{table}
  \centering
  \begin{tabular}{c c c c c c c}
    \toprule
    Number of time steps & $\Delta t$ (s) &  $\|\bm{u}-\bm{u}_{\mathrm{ex}}\|/\|\bm{u}_{\mathrm{ex}}\|$ & rate$_{\bm{u}}$
      & $\|D-D_{\mathrm{ex}}\|/\|D_{\mathrm{ex}}\|$ & rate$_D$ \\
    \midrule
     160 & 2700 & $7.17\times10^{-2}$ & --    & $3.52\times10^{-3}$ & --    \\
     320 & 1350 & $1.80\times10^{-2}$ & 1.99 & $8.84\times10^{-4}$ & 1.99 \\
     640 &  675 & $4.50\times10^{-3}$ & 1.99 & $2.21\times10^{-4}$ & 1.99 \\
    1280 &  338 & $1.13\times10^{-3}$ & 2.00 & $5.55\times10^{-5}$ & 1.99 \\
    \bottomrule
  \end{tabular}
  \caption{Temporal convergence for the L\"auter et al test case. BDM$_3$/DG$_2$,
    mesh degree $2$, degree $1$ in time, fixed mesh (mesh refinement 5), $t = 5$\,days.}
  \label{tab:lauter-temporal-coords2}
\end{table}

Figure~\ref{fig:lauter-div} shows the time evolution of the $L^2$ norm of the velocity divergence. 
In this test problem, the initial condition is in geostrophic balance except for the region
around the mountain. This causes a coherent gravity wave pattern initially, with the waves
dispersing and mixing their phases at later times, at which point the divergence norm
looks flatter. This demonstrates that the discretisation is avoiding spurious coupling between
balanced and unbalanced modes.
Figure~\ref{fig:lauter-energy} shows the relative energy drift,
which stays below $10^{-12}$ for the whole run (this can be further reduced
by choosing even tighter solve tolerances).

\begin{figure}
  \centering
  \includegraphics[width=0.68\textwidth]{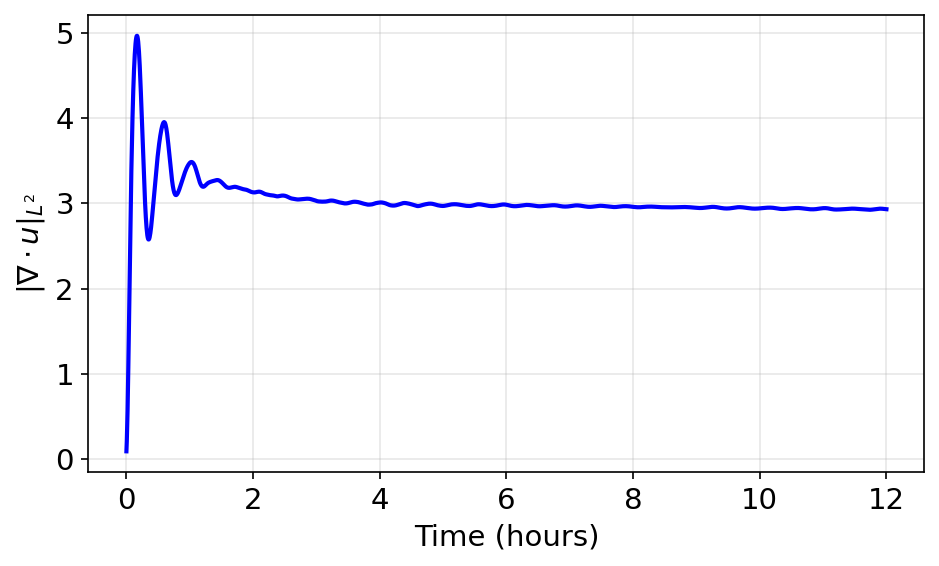}
  \caption{Velocity divergence $\|\nabla\cdot \bm{u}\|_{L^2}$ versus time for the L\"auter et al test case.
    BDM$_2$/DG$_1$, mesh refinement 4, $\Delta t = 42.2$\,s.}
  \label{fig:lauter-div}
\end{figure}

\begin{figure}
  \centering
  \includegraphics[width=0.68\textwidth]{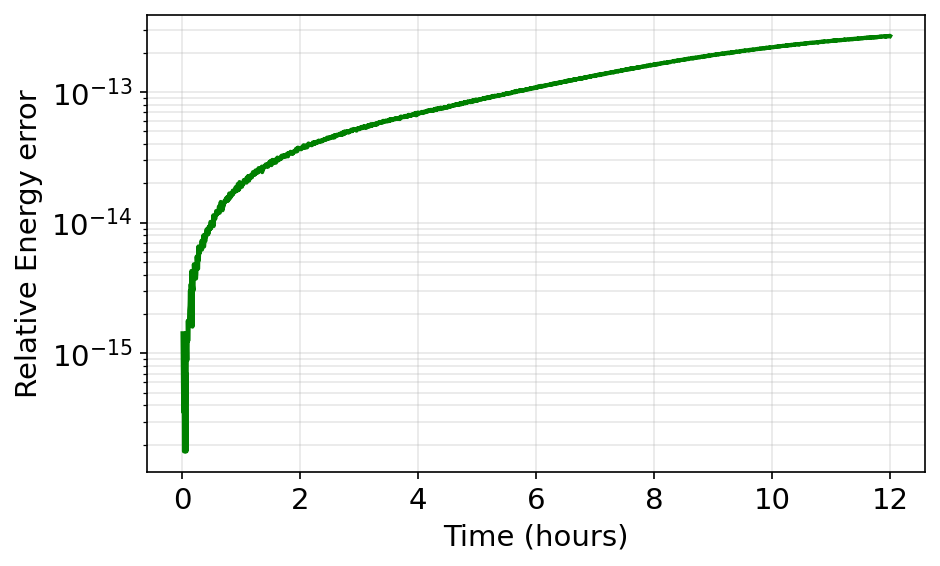}
  \caption{Relative energy drift $|(H-H_0)/H_0|$ versus time for the L\"auter et al test case.
    BDM$_2$/DG$_1$, mesh refinement 4, $\Delta t = 42.2$\,s.}
  \label{fig:lauter-energy}
\end{figure}


For the higher order case we use BDM$_3$/DG$_2$ with degree $2$ in
time.  Table~\ref{tab:lauter-spatial-deg2} gives spatial convergence
rates close to third order, and Table~\ref{tab:lauter-temporal-deg2}
gives temporal rates between 3 and 4.  For degree $s$ PcPG schemes we
expect order $s+1$ convergence, but these schemes are minor
perturbations of Gauss-Legendre schemes, for which we expect order
$2s$ convergence, and we suspect that we have not yet reached the 3rd
order asymptotic regime because of this.

\begin{table}
  \centering
  \begin{tabular}{c c c c c c}
    \toprule
    mesh refinement& $h$ (km) & $\|\bm{u}-\bm{u}_{\mathrm{ex}}\|/\|\bm{u}_{\mathrm{ex}}\|$ & rate$_{\bm{u}}$
        & $\|D-D_{\mathrm{ex}}\|/\|D_{\mathrm{ex}}\|$ & rate$_D$ \\
    \midrule
    2 & 1785.5 & $8.89\times10^{-4}$ & --   & $3.02\times10^{-4}$ & --   \\
    3 &  892.8 & $9.66\times10^{-5}$ & 3.20 & $3.82\times10^{-5}$ & 2.98 \\
    4 &  446.4 & $1.15\times10^{-5}$ & 3.07 & $4.79\times10^{-6}$ & 3.00 \\
    \bottomrule
  \end{tabular}
  \caption{Spatial convergence for the L\"auter et al test case.
    BDM$_3$/DG$_2$, mesh degree $3$, degree $2$ in time, $\Delta t = 43.2$\,s,
    $t = 12$\,h.}
  \label{tab:lauter-spatial-deg2}
\end{table}

\begin{table}
  \centering
  \begin{tabular}{c c c c c c c}
    \toprule
    Number of timesteps & $\Delta t$ (s) & $\|\bm{u}-\bm{u}_{\mathrm{ex}}\|/\|\bm{u}_{\mathrm{ex}}\|$ & rate$_{\bm{u}}$
      & $\|D-D_{\mathrm{ex}}\|/\|D_{\mathrm{ex}}\|$ & rate$_D$ \\
    \midrule
    240 & 5400 & $2.20\times10^{-3}$ & --   & $1.08\times10^{-4}$ & --   \\
    480 & 2700 &  $1.38\times10^{-4}$ & 4.00 & $8.27\times10^{-6}$ & 3.71 \\
    960 & 1350 &  $1.35\times10^{-5}$ & 3.35 & $9.39\times10^{-7}$ & 3.13 \\
    \bottomrule
  \end{tabular}
  \caption{Temporal convergence for the L\"auter et al test case.
    BDM$_3$/DG$_2$, mesh degree $3$, degree $2$ in time, fixed mesh (refinement~4),
    $t = 15$\,days.}
  \label{tab:lauter-temporal-deg2}
\end{table}

\subsection{Test cases without analytic solution}\label{sec:nonanalytic}
In this section, we consider two standard test cases with no closed
form solution for the rotating shallow water equations, where
convergence is measured against the fine grid references described
above: the Rossby--Haurwitz wave, a (near) travelling wave solution,
and a flow over an isolated mountain, where orographic forcing generates
gravity waves and a lee vortex train throughout the integration. Each
is run in both the deterministic and the stochastic setting.

\subsubsection{Rossby--Haurwitz wave}
The test case~6 of \cite{williamson1992standard} is a Rossby--Haurwitz wave of zonal wavenumber~4.
Although the Rossby--Haurwitz wave is an
exact solution of the nondivergent barotropic vorticity equation and not of the
shallow water equations, it is a standard smooth benchmark. The initial velocity field is nondivergent and is obtained from the stream
function
\begin{equation}
  \psi = -a^2\omega\,\sin\theta
         + a^2 K\,\cos^{R}\!\theta\,\sin\theta\,\cos R\lambda ,
  \label{eq:w6-psi}
\end{equation}
where $\lambda$ and $\theta$ denote longitude and latitude, $a$ is the sphere
radius, and $\omega$, $K$, $R$ are constants. The velocity is initialised as
$\bm{u}_0 = \nabla^{\perp}\psi = \hat{\bm{k}}\times\nabla\psi$, projected into the
velocity space $V$; equivalently, the analytic components are
\begin{align}
  u &= a\omega\cos\theta
       + aK\,\cos^{R-1}\!\theta\,(R\sin^2\theta - \cos^2\theta)\,\cos R\lambda ,
       \label{eq:w6-u}\\
  v &= -aKR\,\cos^{R-1}\!\theta\,\sin\theta\,\sin R\lambda .
       \label{eq:w6-v}
\end{align}
The depth field is obtained by solving the balance equation so that the initial
divergence tendency vanishes,
\begin{equation}
  D = h_0 + \frac{a^2}{g}\left( A(\theta) + B(\theta)\cos R\lambda
             + C(\theta)\cos 2R\lambda \right) ,
  \label{eq:w6-D}
\end{equation}
with
\begin{align}
  A(\theta) &= \tfrac12\,\omega(2\Omega + \omega)\cos^2\theta
    + \tfrac14 K^2 \cos^{2R} \theta
      \left[(R+1)\cos^2\theta + (2R^2 - R - 2)
           - 2R^2\cos^{-2}\theta\right], \label{eq:w6-A}\\
  B(\theta) &= \frac{2(\Omega + \omega)K}{(R+1)(R+2)}\,
    \cos^{R}\theta
      \left[(R^2 + 2R + 2) - (R+1)^2\cos^2\theta\right], \label{eq:w6-B}\\
  C(\theta) &= \tfrac14 K^2 \cos^{2R}\theta
      \left[(R+1)\cos^2\theta - (R+2)\right]. \label{eq:w6-C}
\end{align}
Here $\Omega = 7.292\times10^{-5}~\mathrm{s^{-1}}$ is the planetary rotation
rate and $g = 9.8~\mathrm{m\,s^{-2}}$. The parameters are
$\omega = K = 7.847\times10^{-6}~\mathrm{s^{-1}}$, $h_0 = 8000~\mathrm{m}$ and
zonal wavenumber $R = 4$, with sphere radius $a = 6{,}371{,}220~\mathrm{m}$. 

We integrate to $t = 2$\,days, comparing against numerical
references. For BDM$_2$/DG$_1$ spaces and degree 1 in time,
Table~\ref{tab:w6-spatial} demonstrates the expected second order
convergence in space (approaching from above, in the case of
velocity), whilst Table~\ref{tab:w6-temporal} demonstrates second
order convergence in time.  Figure~\ref{fig:w6-div} shows the velocity
divergence. Initially divergence free, the velocity develops only a
small bounded divergence with no secular growth, $\|\nabla\cdot
\bm{u}\|/\|\bm{u}\| \sim 10^{-7}$.  Figure~\ref{fig:w6-energy}
confirms the energy conservation.

\begin{table}
  \centering
  \begin{tabular}{c c c c c c}
    \toprule
    mesh refinement& $h$ (km) & $\|\bm{u}-\bm{u}_{\mathrm{ref}}\|_{L^2}$ & rate$_{\bm{u}}$
        & $\|D-D_{\mathrm{ref}}\|_{L^2}$ & rate$_D$ \\
    \midrule
    2 & 1785.5 & $1.28\times10^{7}$ & --   & $5.45\times10^{8}$ & --   \\
    3 &  892.8 & $1.75\times10^{6}$ & 2.87 & $1.39\times10^{8}$ & 1.97 \\
    4 &  446.4 & $3.01\times10^{5}$ & 2.54 & $3.47\times10^{7}$ & 2.00 \\
    5 &  223.2 & $5.51\times10^{4}$ & 2.45 & $8.41\times10^{6}$ & 2.04 \\
    \bottomrule
  \end{tabular}
  \caption{Spatial convergence for the deterministic Rossby-Haurwitz wave. BDM$_2$/DG$_1$,
    mesh degree $2$, degree $1$ in time, $t = 2$\,days. $L^2$ differences against
    a numerical reference at mesh refinement 6.}
  \label{tab:w6-spatial}
\end{table}

\begin{table}
  \centering
  \begin{tabular}{c c c c c c c}
    \toprule
    Number of timesteps & $\Delta t$ (s) 
      & $\|\bm{u}-\bm{u}_{\mathrm{ref}}\|_{L^2}$ & rate$_{\bm{u}}$
      & $\|D-D_{\mathrm{ref}}\|_{L^2}$ & rate$_D$ \\
    \midrule
     80 & 2160 &  $9.46\times10^{6}$ & --   & $1.61\times10^{8}$ & --   \\
    160 & 1080 & $4.70\times10^{6}$ & 1.01 & $1.24\times10^{8}$ & 0.38 \\
    320 &  540 &  $1.28\times10^{6}$ & 1.88 & $5.55\times10^{7}$ & 1.16 \\
    640 &  270 &  $3.42\times10^{5}$ & 1.90 & $1.34\times10^{7}$ & 2.05 \\
    \bottomrule
  \end{tabular}
  \caption{Temporal convergence for the deterministic Rossby-Haurwitz wave. BDM$_2$/DG$_1$,
    degree $1$ in time, fixed mesh (mesh refinement 5), $t = 2$\,days. $L^2$ differences against a
    small timestep reference with $2560$ steps.}
  \label{tab:w6-temporal}
\end{table}

\begin{figure}
  \centering
  \includegraphics[width=0.65\textwidth]{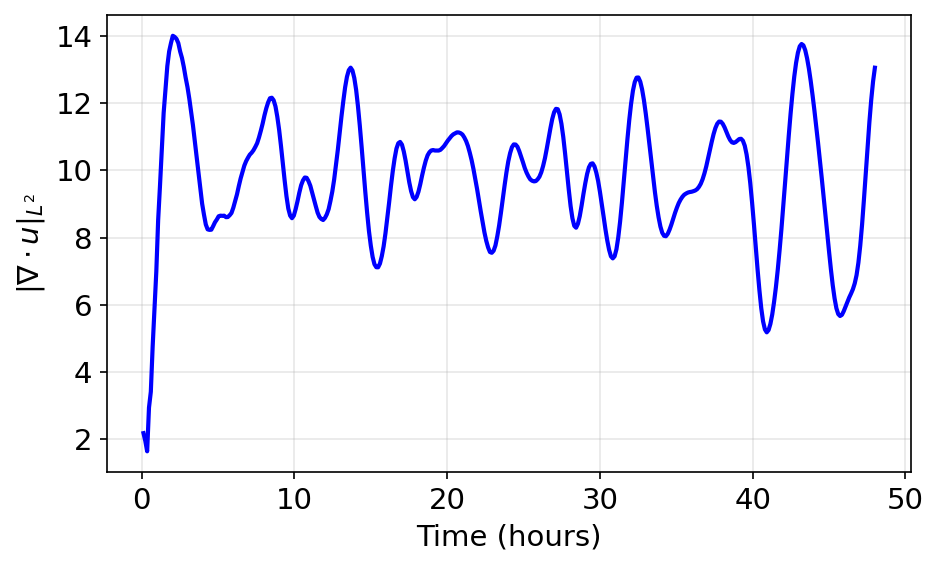}
  \caption{Velocity divergence $\|\nabla\cdot \bm{u}\|_{L^2}$ versus time
    for the deterministic Rossby-Haurwitz wave. BDM$_2$/DG$_1$, mesh refinement 5, $\Delta t = 67.5$\,s.}
  \label{fig:w6-div}
\end{figure}

\begin{figure}
  \centering
  \includegraphics[width=0.65\textwidth]{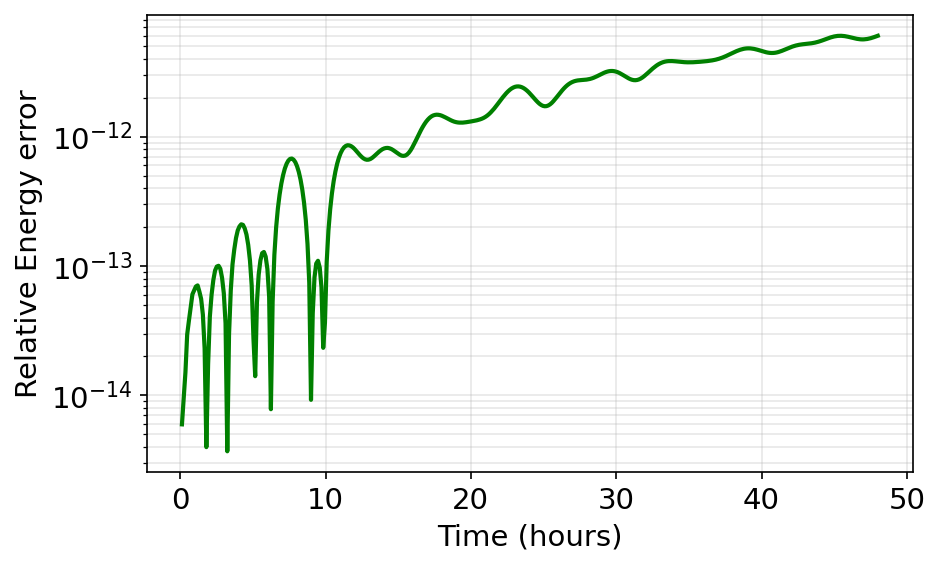}
  \caption{Relative energy drift $|(H-H_0)/H_0|$ versus time
    for the deterministic Rossby-Haurwitz wave. BDM$_2$/DG$_1$, mesh refinement 5, $\Delta t = 67.5$\,s.}
  \label{fig:w6-energy}
\end{figure}

\subsubsection{Stochastic Rossby-Haurwitz wave}
We now add SFLT energy-conserving transport noise to the
Rossby-Haurwitz wave test, using correlation length $\lambda =
6\times10^6\,\mathrm{m}$ and noise level $r_\sigma = 0.10$ (as defined
in the Appendix).

Table~\ref{tab:w6-sflt-robust} reports the solver behaviour in the
presence of SFLT noise.  The average number of GMRES iterations per
timestep remains at around 11 across the refinement levels (recalling
the scaling of $\Delta t$ with $\Delta x$). The energy drift is also
under control, being due to truncation of iterative solvers.



\begin{table}
  \centering
  \begin{tabular}{c c c c c}
    \toprule
    mesh refinement& $h$ (km) & $\Delta t$ (s) & avg.\ linear iters / step & $|(H-H_0)/H_0|$ \\
    \midrule
    2 & 1785.5 & 3456 & 11.38 & $2.1\times10^{-8}$  \\
    3 &  892.8 & 1728 & 11.00 & $1.6\times10^{-9}$  \\
    4 &  446.4 &  864 & 11.00 & $9.6\times10^{-11}$ \\
    5 &  223.2 &  432 & 11.82 & $6.0\times10^{-12}$ \\
    \bottomrule
  \end{tabular}
  \caption{SFLT iterative solver behaviour and energy conservation for the Rossby-Haurwitz wave
    testcase.
    BDM$_2$/DG$_1$, degree $1$ in time, $\Delta t \propto h$ (Number of timesteps$=50,100,200,400$
    for refinement $2$--$5$), $t = 2$\,days.}
  \label{tab:w6-sflt-robust}
\end{table}

In these experiments, we measure the pathwise error $\|z_\ell -
z_{\mathrm{ref}}\|$ at fixed noise realisation. Let $\mathcal{T}_\ell$
denote $\mathcal{T}_h$ at level~$\ell$, topologically nested but not
geometrically nested due to the approximate polynomial representation
of the sphere. In space, $\psi_{\mathrm{noise}}$ is sampled once on
the finest mesh and $L^2$-projected onto
$\mathrm{CG}_2(\mathcal{T}_\ell)$, yielding a velocity noise
$\delta\bm{B}^n_\ell$ at each resolution. In time, the finest run
draws the increments $\Delta W_n$ on its step $\Delta t_{\min}$, and a
coarse step of length $m\,\Delta t_{\min}$ sums them,
$\sum_{j=1}^{m}\Delta W_{n+j}$, of variance $m\,\Delta t_{\min} =
\Delta t$. The finest run thus fixes one Brownian path that every
coarser run subsamples.

Table~\ref{tab:w6-sflt-spatial} gives the spatial convergence. The
rates are similar to the deterministic ones in
Table~\ref{tab:w6-spatial}. Table~\ref{tab:w6-sflt-temporal} gives the
temporal convergence. At $r_\sigma=0.1$, both fields converge at
approximately first order in $\Delta t$, as expected.

\begin{table}
  \centering
  \begin{tabular}{c c c c c c}
    \toprule
    mesh refinement & $h$ (km) & $\|\bm{u}-\bm{u}_{\mathrm{ref}}\|_{L^2}$ & rate$_{\bm{u}}$
        & $\|D-D_{\mathrm{ref}}\|_{L^2}$ & rate$_D$ \\
    \midrule
    2 & 1785.5 & $1.27\times10^{7}$ & --   & $5.45\times10^{8}$ & --   \\
    3 &  892.8 & $1.74\times10^{6}$ & 2.87 & $1.38\times10^{8}$ & 1.98 \\
    4 &  446.4 & $2.89\times10^{5}$ & 2.59 & $3.36\times10^{7}$ & 2.04 \\
    \bottomrule
  \end{tabular}
  \caption{SFLT spatial (strong) convergence for the Rossby-Haurwitz wave. BDM$_2$/DG$_1$, degree $1$ in time,
    $t = 2$\,days, averaged over 3 seeds.}
  \label{tab:w6-sflt-spatial}
\end{table}


\begin{table}
  \centering
  \begin{tabular}{c c c c c c}
    \toprule
    Number of timesteps & $\Delta t$ (s) & $\|\bm{u}-\bm{u}_{\mathrm{ref}}\|_{L^2}$ & rate$_{\bm{u}}$
        & $\|D-D_{\mathrm{ref}}\|_{L^2}$ & rate$_D$ \\
    \midrule
    200  & 864 & $1.79\times10^{7}$ & --   & $1.87\times10^{8}$ & --   \\
    400  & 432 & $9.77\times10^{6}$ & 0.87 & $1.01\times10^{8}$ & 0.90 \\
    800  & 216 & $4.07\times10^{6}$ & 1.26 & $4.48\times10^{7}$ & 1.17 \\
    \bottomrule
  \end{tabular}
  \caption{SFLT temporal (strong) convergence for the Rossby-Haurwitz wave. BDM$_2$/DG$_1$, degree $1$ in time, fixed
    mesh (mesh refinement 4), $t = 2$\,days. Errors compared against a reference run
    with 1600 timesteps, then averaged over 3 seeds for the random number generator.}
  \label{tab:w6-sflt-temporal}
\end{table}

Figures~\ref{fig:w6-sflt-energy} and \ref{fig:w6-sflt-div} show that
the discretised SFLT noise does not affect energy conservation, and
that the SFLT model is not generating significant unbalanced motion as
measured by the divergence.  Figures~\ref{fig:w6-compare-vort_small}
and \ref{fig:w6-compare-vort} show the effect of the SFLT noise on the
solution. In Figure~\ref{fig:w6-compare-vort} the amplitude is raised
to $r_\sigma \approx 0.5$, above the reference value $\approx 0.1$, to
make the structure visible: the smooth wavenumber-$4$ Rossby--Haurwitz
wave acquires small scale eddies with the large scale pattern intact.

\begin{figure}
  \centering
  \includegraphics[width=0.65\textwidth]{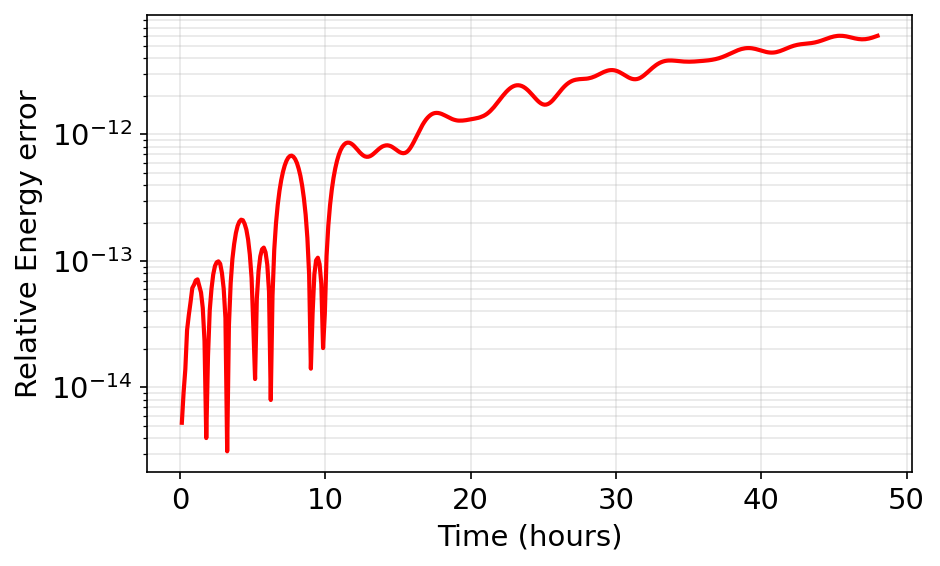}
  \caption{SFLT relative energy drift $|(H-H_0)/H_0|$ versus time  for the Rossby-Haurwitz wave.
    BDM$_2$/DG$_1$, mesh refinement 5, degree $1$ in time, $\Delta t = 432$\,s.}
  \label{fig:w6-sflt-energy}
\end{figure}

\begin{figure}
  \centering
  \includegraphics[width=0.65\textwidth]{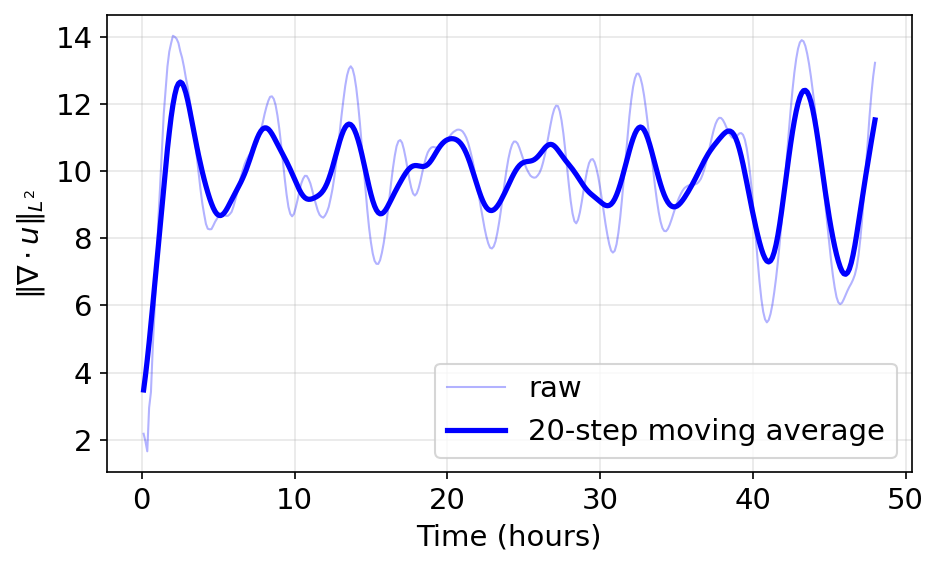}
  \caption{SFLT velocity divergence $\|\nabla\cdot \bm{u}\|_{L^2}$ versus
    time  for the Rossby-Haurwitz wave. BDM$_2$/DG$_1$, mesh refinement 5, degree $1$ in time, $\Delta t = 432$\,s.}
  \label{fig:w6-sflt-div}
\end{figure}


\begin{figure}
  \centering
  \includegraphics[width=\textwidth]{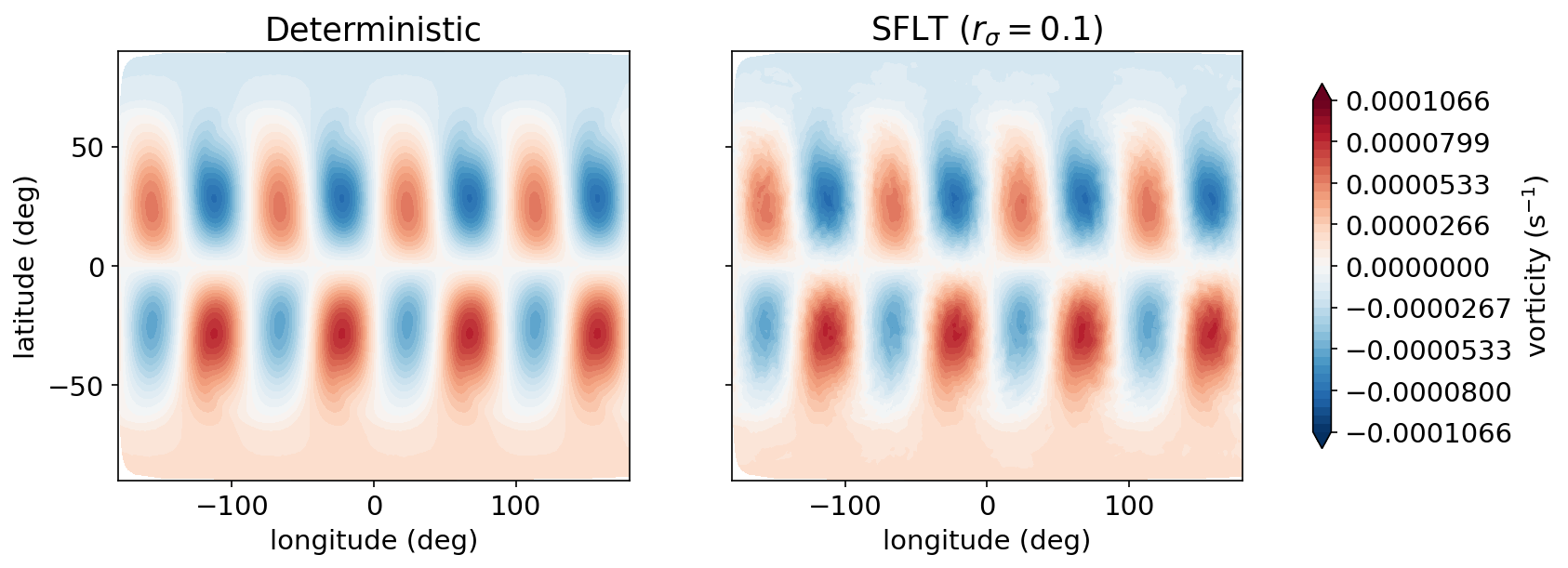}
  \caption{Rossby-Haurwitz wave vorticity at $t = 2$\,days, deterministic versus
    SFLT at mild forcing ($r_{\sigma} \approx 0.1$). Shared colour
    scale. BDM$_2$/DG$_1$, mesh refinement 5, degree $1$ in time.}
  \label{fig:w6-compare-vort_small}
\end{figure}

\begin{figure}
  \centering
  \includegraphics[width=\textwidth]{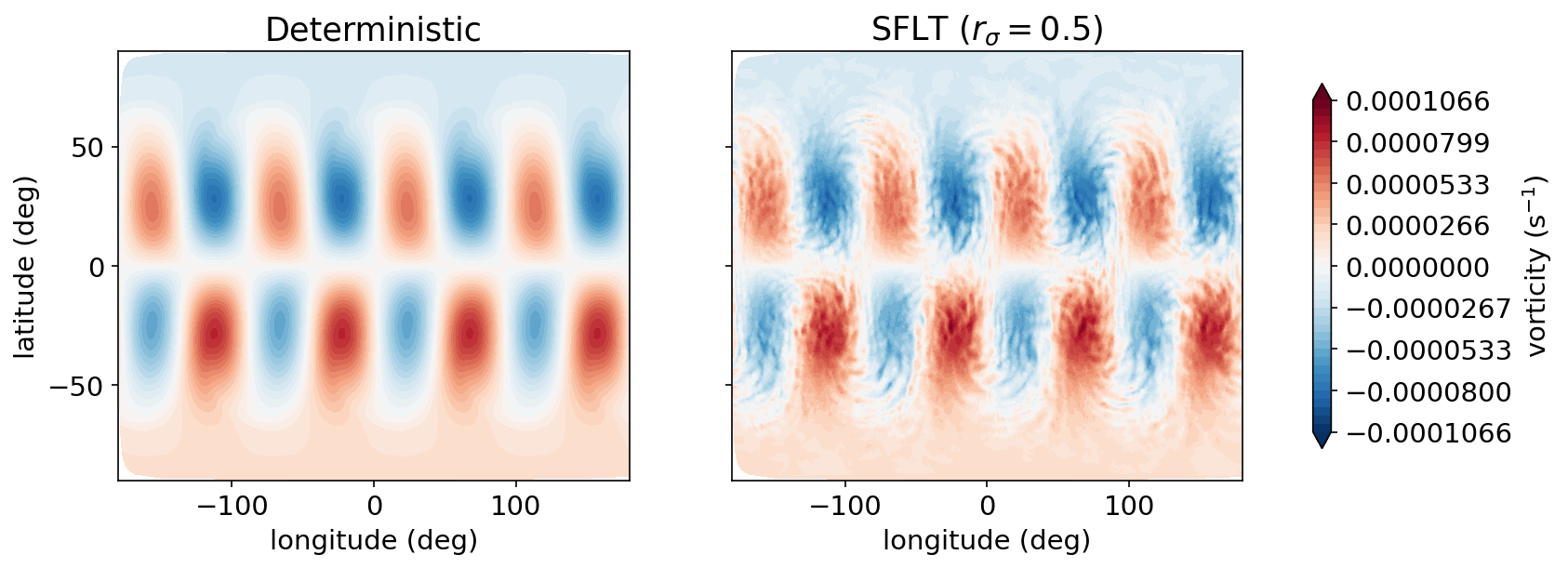}
  \caption{Rossby-Haurwitz wave vorticity at $t = 2$\,days, deterministic versus
    SFLT at strong forcing ($r_{\sigma} \approx 0.5$). Shared colour
    scale. BDM$_2$/DG$_1$, mesh refinement 5, degree $1$ in time.}
  \label{fig:w6-compare-vort}
\end{figure}

To examine noise of larger spatial scale, we increase the correlation
length to $\lambda = 3\times10^{6}\,\mathrm{m}$ with $r_\sigma \approx
0.5$ and integrate for $5$ days
(Figures~\ref{fig:w6-lam3e6-day1}--\ref{fig:w6-lam3e6-day5}). The
large-scale features remain dominated by the deterministic flow: the
wavenumber-$4$ pattern is retained, and the SFLT run departs from the
deterministic one through a smooth difference that grows in time.

\begin{figure}
  \centering
  \includegraphics[width=\textwidth]{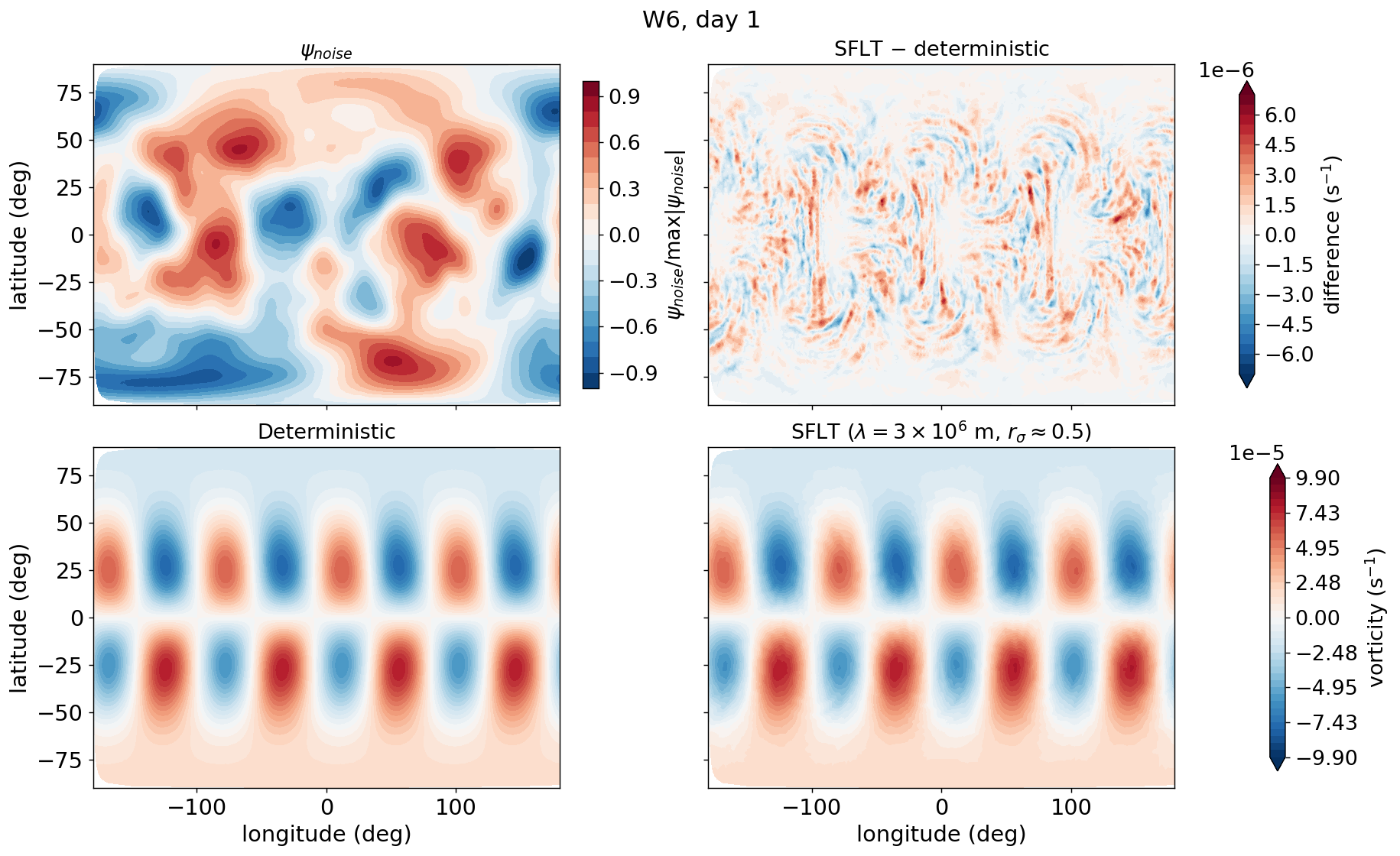}
  \caption{Day~$1$  for the Rossby-Haurwitz wave, $\lambda = 3\times10^{6}\,\mathrm{m}$, $r_\sigma \approx 0.5$.
  Top: $\psi_{\mathrm{noise}}$ (left) and SFLT minus deterministic vorticity
  (right). Bottom: deterministic (left) and SFLT (right) vorticity.}
\label{fig:w6-lam3e6-day1}
\end{figure}

\begin{figure}
  \centering
  \includegraphics[width=\textwidth]{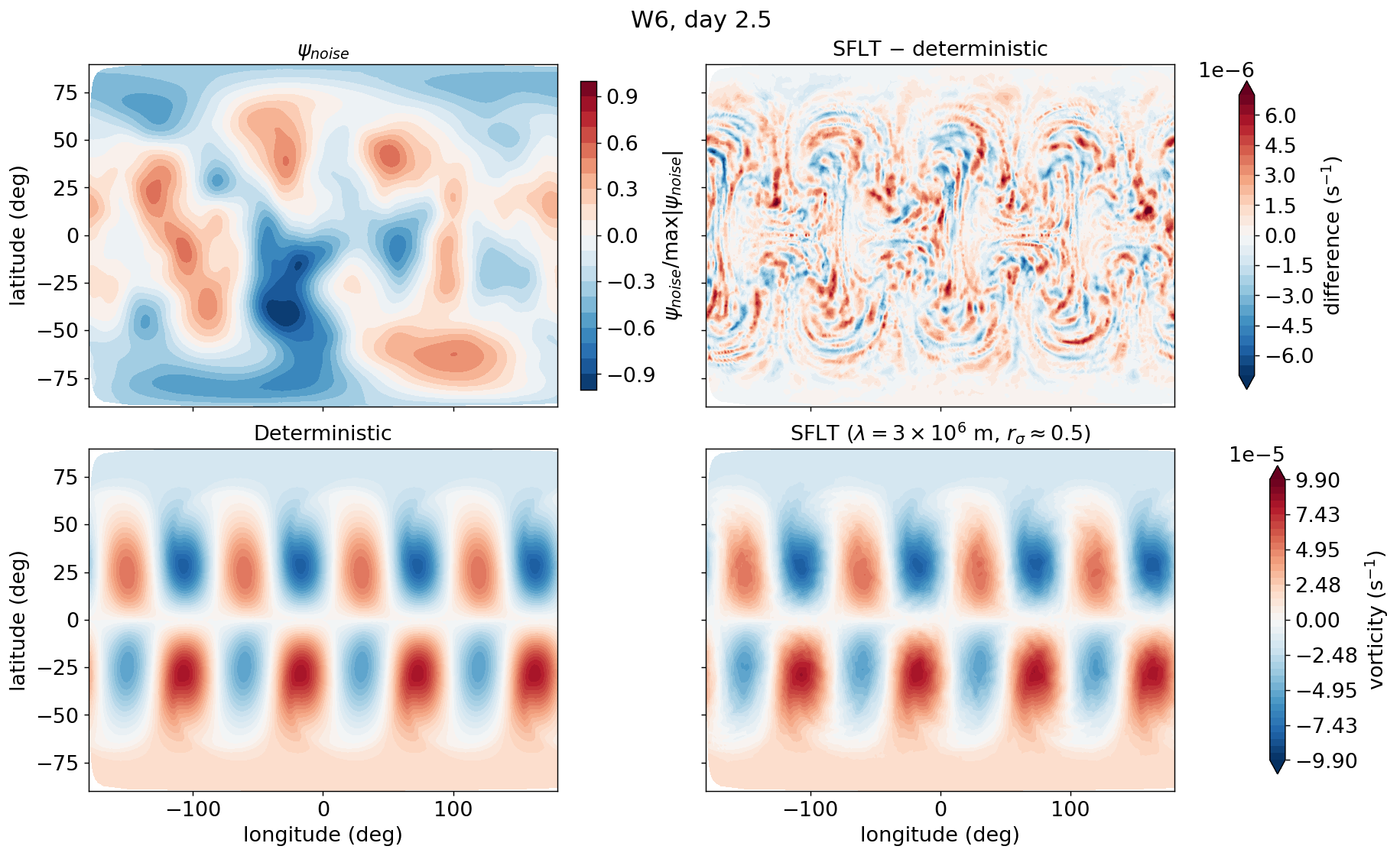}
  \caption{As Figure~\ref{fig:w6-lam3e6-day1}, day~$2.5$.}
  \label{fig:w6-lam3e6-day2.5}
\end{figure}

\begin{figure}
  \centering
  \includegraphics[width=\textwidth]{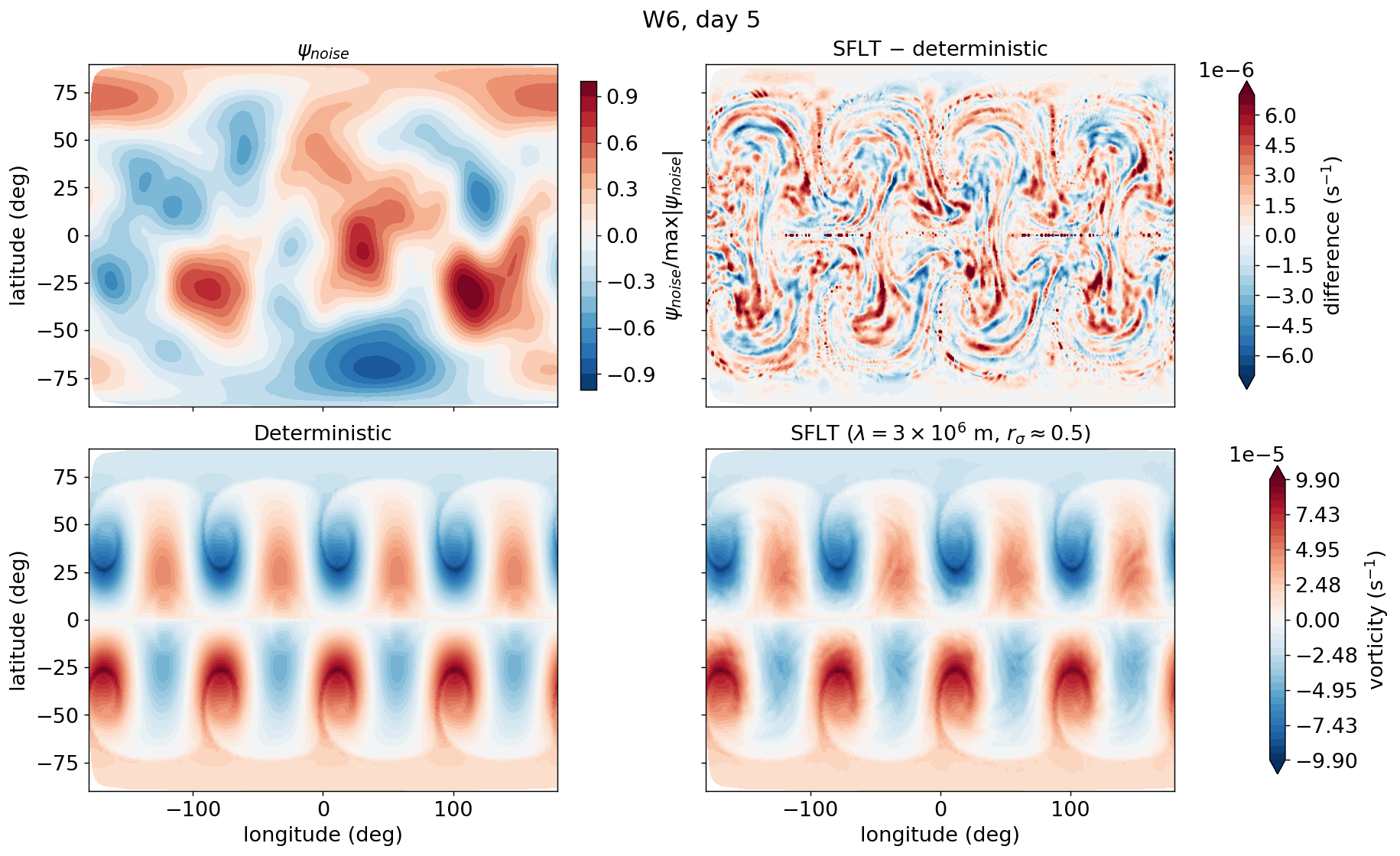}
  \caption{As Figure~\ref{fig:w6-lam3e6-day1}, day~$5$.}
  \label{fig:w6-lam3e6-day5}
\end{figure}

Figure~\ref{fig:sflt-amp-sweep} varies the amplitude from $0$ to $0.48$ of the
flow speed at fixed correlation length, demonstrating robustness of the
energy conservation to the SFLT noise (the curves appear superposed at
this scale).
Figure~\ref{fig:sflt-ensemble-energy} repeats the physical-amplitude run over
$16$ seeds. The per-seed maxima are indistinguishable, so conservation is
seed-independent and holds pathwise, not only in expectation.

\begin{figure}
  \centering
  \includegraphics[width=0.65\textwidth]{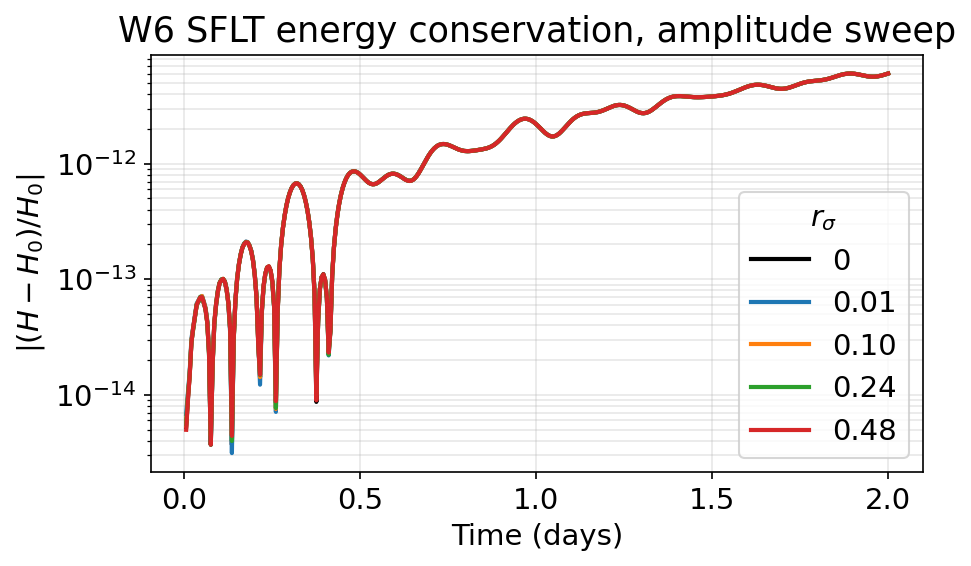}
  \caption{SFLT relative energy drift $|(H-H_0)/H_0|$  for the Rossby-Haurwitz wave under varying noise
    levels $r_\sigma \approx \{0,\,0.01,\,0.10,\,0.24,\,0.48\}$.}
  \label{fig:sflt-amp-sweep}
\end{figure}

\begin{figure}
  \centering
  \includegraphics[width=0.65\textwidth]{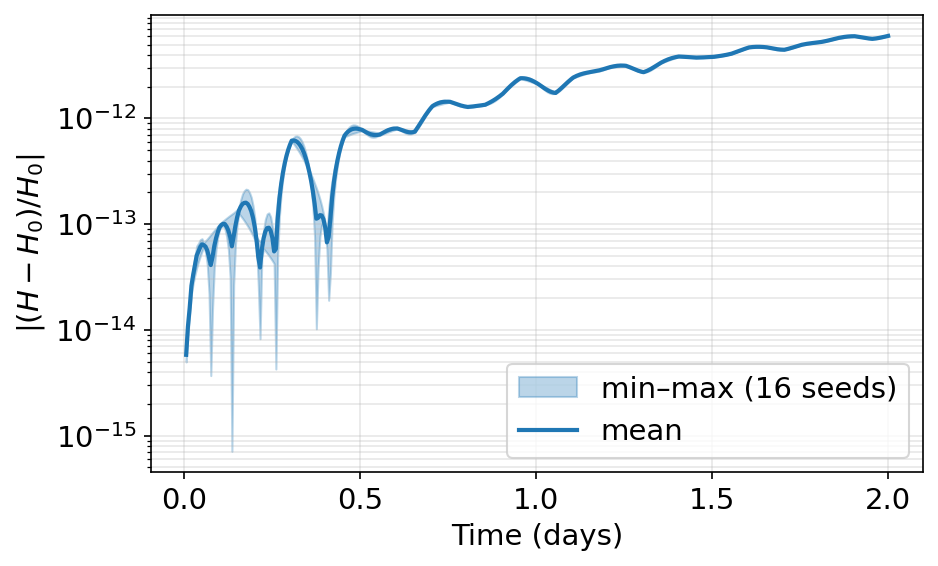}
  \caption{SFLT relative energy drift $|(H-H_0)/H_0|$  for the Rossby-Haurwitz wave across an ensemble of
    $16$ seeds at physical amplitude ($r_{\sigma} \approx 0.10$).}
  \label{fig:sflt-ensemble-energy}
\end{figure}

\subsubsection{Flow over an isolated mountain}
Test case 5 of  \cite{williamson1992standard}  prescribes a steady solid body rotation zonal flow
impinging on an isolated mountain. There is no analytic solution. The
orography generates gravity waves and a downstream Rossby wave train throughout
the integration, so the dynamics remain active. With Cartesian coordinates
$\bm{x}=(x,y,z)$ on the sphere of radius~$a$, longitude
$\lambda=\operatorname{atan2}(y,x)$ and latitude $\varphi=\arcsin(z/a)$, the
initial velocity is an eastward zonal jet of peak speed~$u_0$,
\begin{equation}
  \bm{u}_0 = \frac{u_0}{a}\,(-y,\,x,\,0) = u_0\cos\varphi\,\hat{\bm{\lambda}},
  \label{eq:w5-u}
\end{equation}
in geostrophic balance with the free-surface elevation
\begin{equation}
  \eta_0(\varphi) = -\frac{1}{g}\left(a\Omega u_0 + \tfrac12 u_0^2\right)\sin^2\varphi,
  \label{eq:w5-eta}
\end{equation}
and the conical mountain is
\begin{equation}
  b(\lambda,\varphi) = h_s\left(1 - \frac{r}{R}\right),
  \quad
  r = \min \left(R,\,\sqrt{(\varphi-\varphi_c)^2 + (\lambda-\lambda_c)^2}\right),
  \label{eq:w5-mtn}
\end{equation}
of height $h_s$ and base radius~$R$ centred at $(\lambda_c,\varphi_c)$, zero outside
$r=R$. The layer depth is $D_0 = H + \eta_0 - b$ with mean depth~$H$.
The sphere radius is $a = 6{,}371{,}220$\,m, the rotation rate
$\Omega = 7.292\times10^{-5}~\mathrm{s^{-1}}$ and
$g = 9.8~\mathrm{m\,s^{-2}}$. The jet speed is $u_0 = 20~\mathrm{m\,s^{-1}}$ and
the mean depth $H = 5960$\,m. The mountain has height $h_s = 2000$\,m and base
radius $R = \pi/9$, centred at $(\lambda_c,\varphi_c) = (-\pi/2,\,\pi/6)$.

Figure~\ref{fig:w5-energy-longtime} shows the relative energy drift over $15$ days,
which is minimal and entirely due to truncation of the iterative solvers after
a tolerance is reached.
Figure~\ref{fig:w5-div-longtime} shows the corresponding divergence: after the
initial adjustment it settles into a bounded oscillation, with no secular growth.


\begin{figure}
  \centering
  \includegraphics[width=0.65\textwidth]{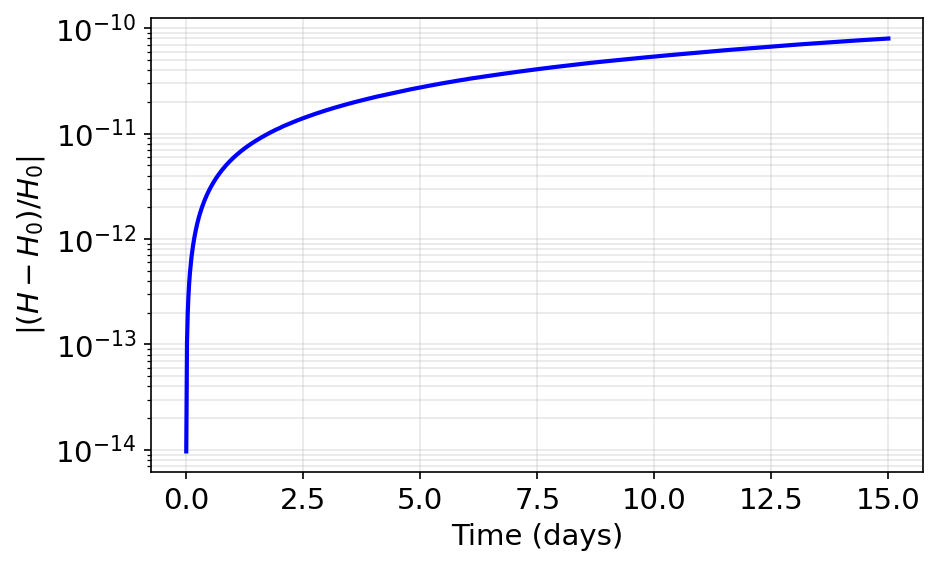}
  \caption{Long-time relative energy drift $|(H-H_0)/H_0|$ versus time for the mountain testcase.
    BDM$_2$/DG$_1$, degree $1$ in time, mesh refinement 5, $\Delta t = 135$\,s.}
  \label{fig:w5-energy-longtime}
\end{figure}

\begin{figure}
  \centering
  \includegraphics[width=0.65\textwidth]{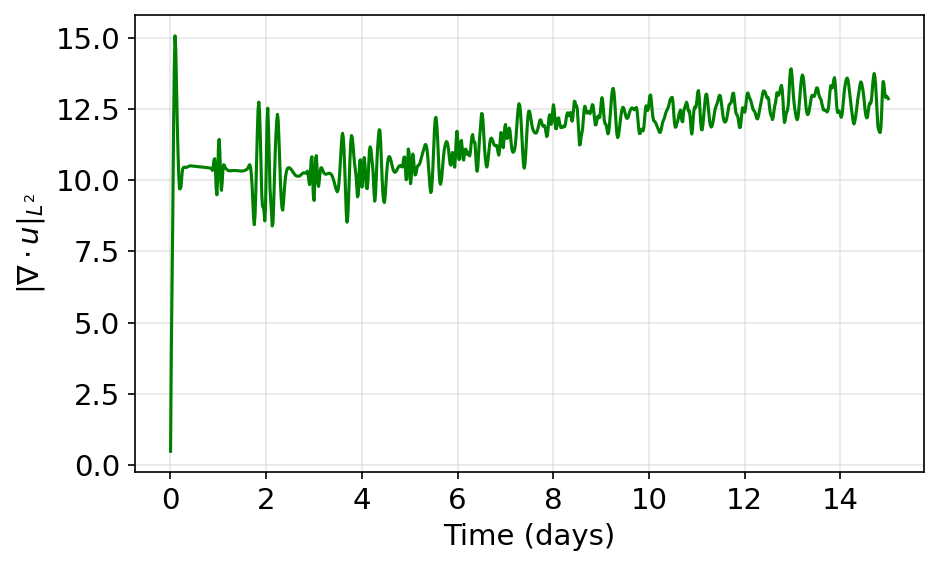}
  \caption{Velocity divergence $\|\nabla\cdot \bm{u}\|_{L^2}$ versus time for
    the run of Figure~\ref{fig:w5-energy-longtime}.}
  \label{fig:w5-div-longtime}
\end{figure}


\subsubsection{Flow over an isolated mountain: stochastic version}

We repeat the forward run with SFLT forcing. The energy
drift over $5$ days stays at round-off with no secular growth
(Figure~\ref{fig:w5-sflt-energy}), matching the deterministic run.  The noise raises the velocity divergence above the deterministic baseline
(Figure~\ref{fig:w5-sflt-div}), and the scheme runs stably with no filtering or
added dissipation. At larger amplitude, the SFLT forcing excites small scale eddies on the
lee  wave and Rossby wave train vorticity pattern while preserving the large scale
structure (Figure~\ref{fig:w5-compare-vort}).

\begin{figure}
  \centering
  \includegraphics[width=0.65\textwidth]{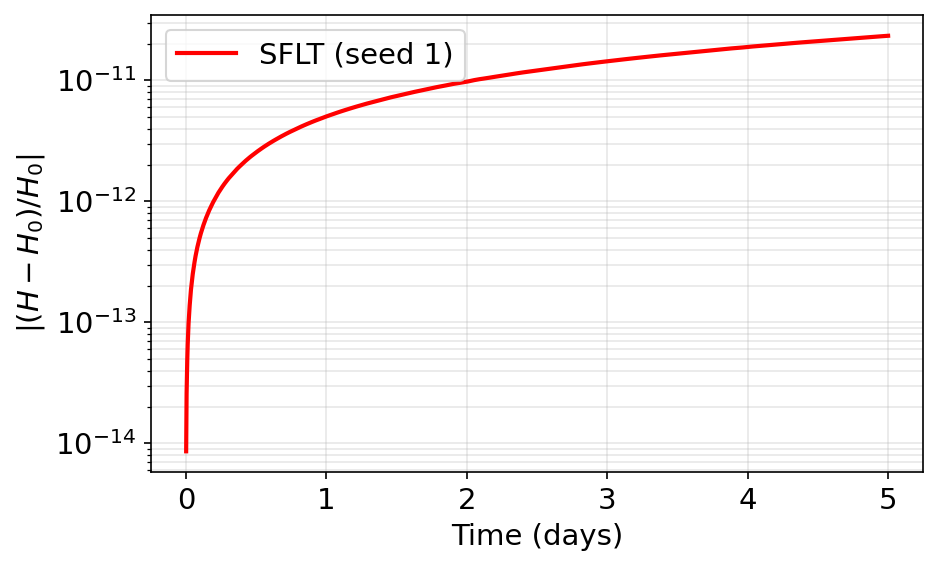}
  \caption{SFLT relative energy drift $|(H-H_0)/H_0|$ over $t=5$ days
    for the mountain testcase (seed~1). BDM$_2$/DG$_1$, mesh refinement 5, degree $1$ in time, $\Delta t = 135$\,s.}
  \label{fig:w5-sflt-energy}
\end{figure}
\begin{figure}
  \centering
  \includegraphics[width=0.65\textwidth]{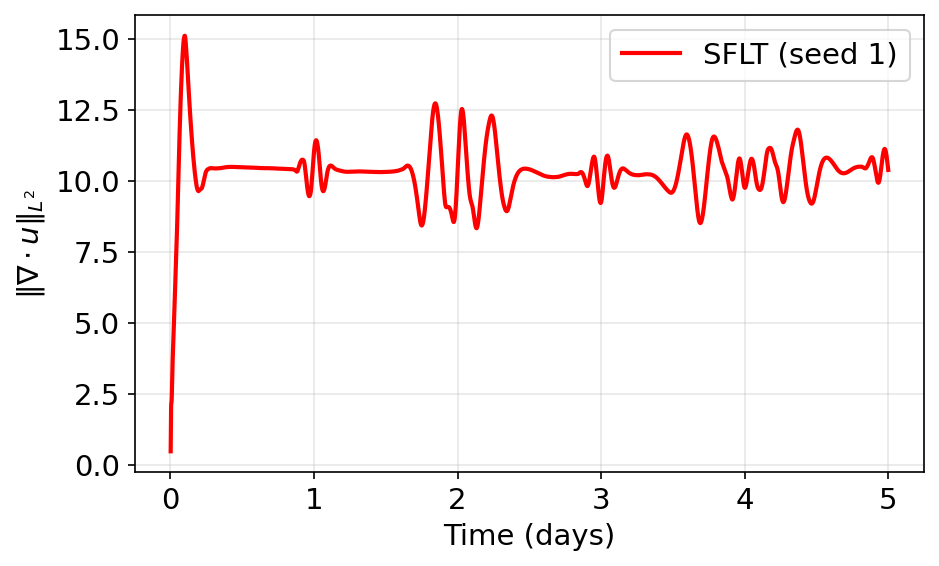}
  \caption{SFLT velocity divergence $\|\nabla\cdot \bm{u}\|_{L^2}$ for the run
    of Figure~\ref{fig:w5-sflt-energy}. BDM$_2$/DG$_1$, mesh refinement 5, degree $1$ in time,
    $\Delta t = 135$\,s.}
  \label{fig:w5-sflt-div}
\end{figure}

\begin{figure}
  \centering
  \includegraphics[width=\textwidth]{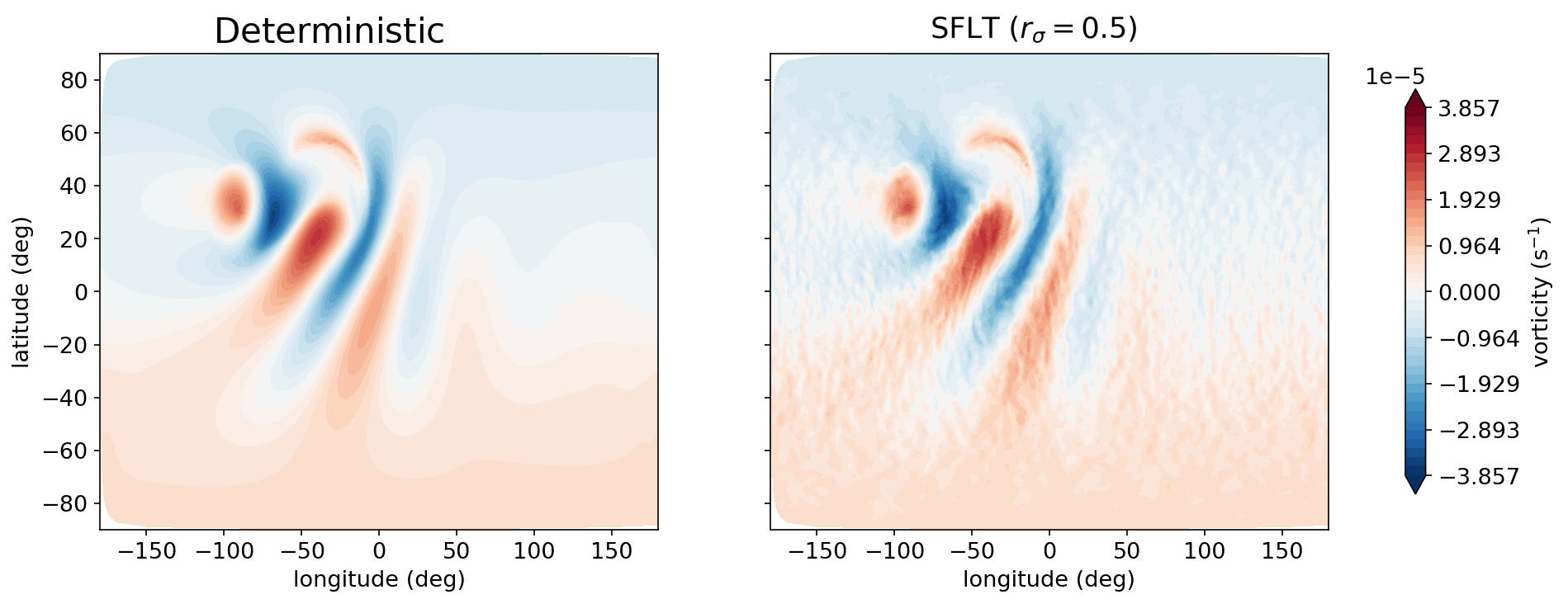}
  \caption{Vorticity at $t = 5$\,days, deterministic versus SFLT at strong
    forcing ($r_{\sigma} \approx 0.5$) for the mountain testcase, with the same colour scale for both images.
    BDM$_2$/DG$_1$, mesh refinement 5, degree $1$ in time.}
  \label{fig:w5-compare-vort}
\end{figure}

As for the Rossby--Haurwitz wave, we repeat the run with larger-scale
noise, $\lambda = 3\times10^{6}\,\mathrm{m}$, here at the stronger
level $r_\sigma \approx 2.0$ and for $9$ days
(Figures~\ref{fig:w5-lam3e6-day1}--\ref{fig:w5-lam3e6-day9}). The SFLT
forcing superimposes small-scale structure on the flow, while the lee
vortex and Rossby wave train of the deterministic run are
retained. The energy error of the SFLT run follows the deterministic
one throughout (Figure~\ref{fig:w5-r2sflt-energy}), and its divergence
fluctuates about the deterministic level without secular growth
(Figure~\ref{fig:w5-r2sflt-div}).

\begin{figure}
  \centering
  \includegraphics[width=\textwidth]{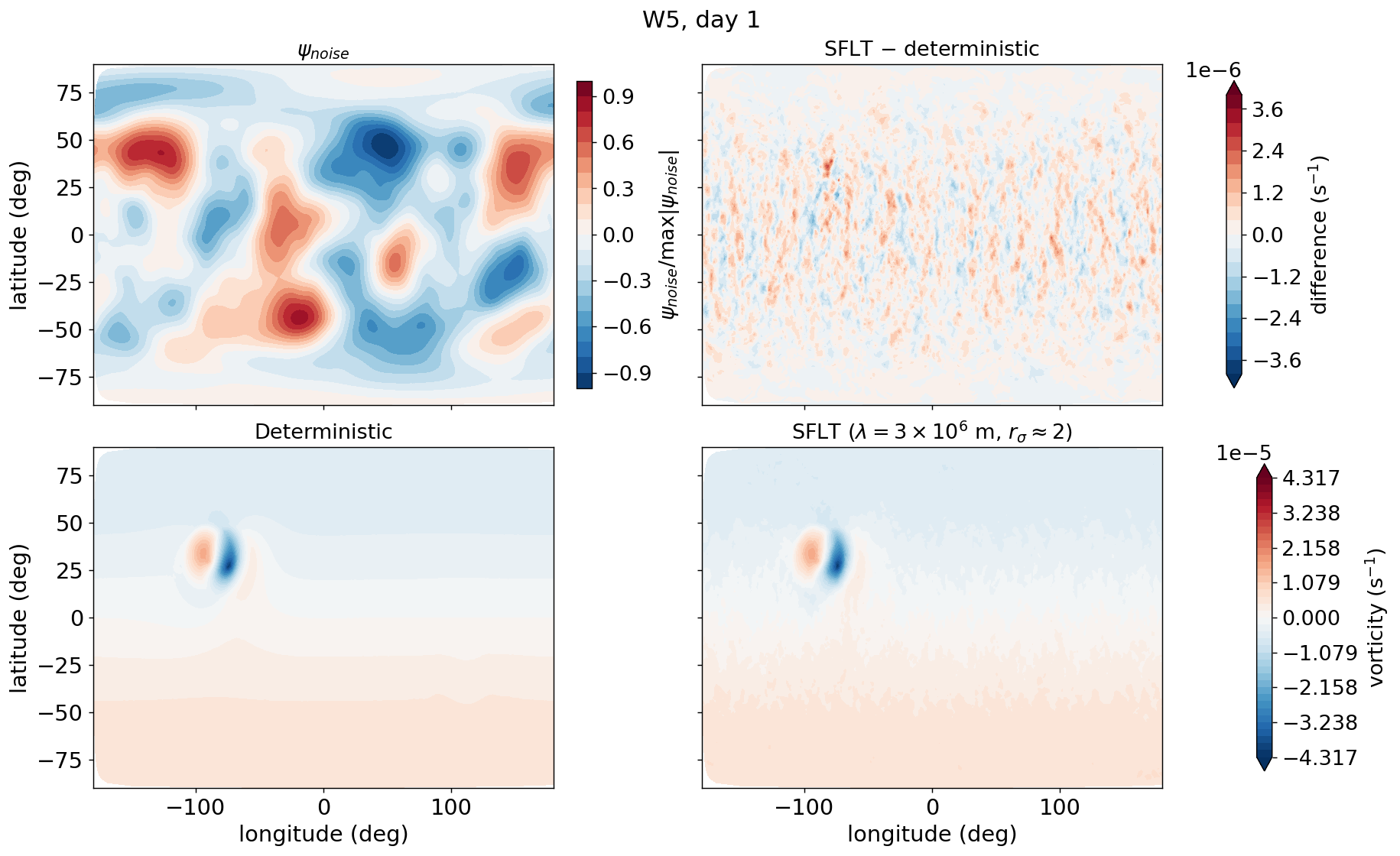}
 \caption{Day~$1$, $\lambda = 3\times10^{6}\,\mathrm{m}$, $r_\sigma \approx 2$.
   Top: $\psi_{\mathrm{noise}}$ (left) and SFLT minus deterministic vorticity
  (right) for the mountain testcase. Bottom: deterministic (left) and SFLT (right) vorticity.}

  \label{fig:w5-lam3e6-day1}
\end{figure}

\begin{figure}
  \centering
  \includegraphics[width=\textwidth]{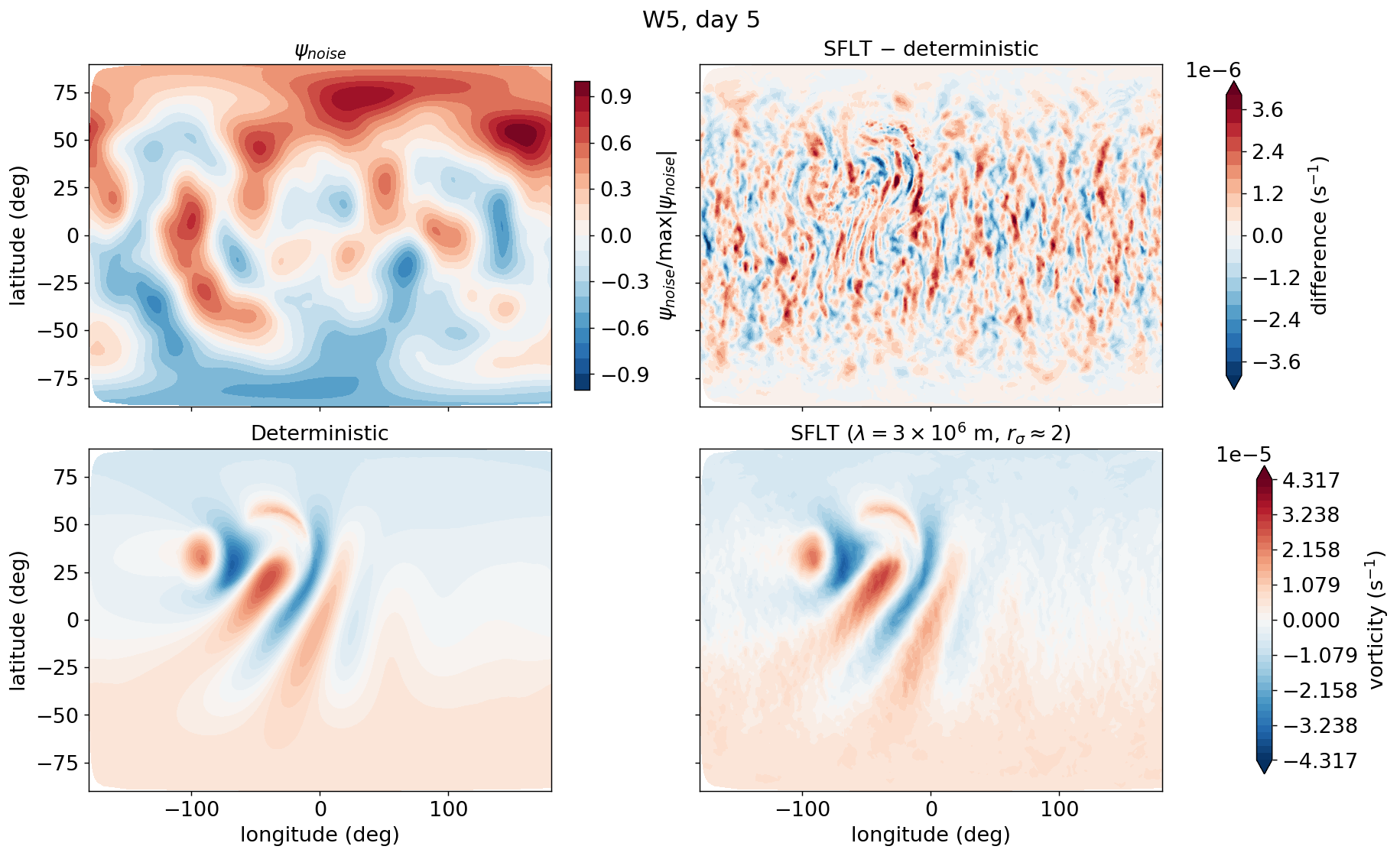}
  \caption{As Figure~\ref{fig:w5-lam3e6-day1}, day~$5$.}
  \label{fig:w5-lam3e6-day5}
\end{figure}

\begin{figure}
  \centering
  \includegraphics[width=\textwidth]{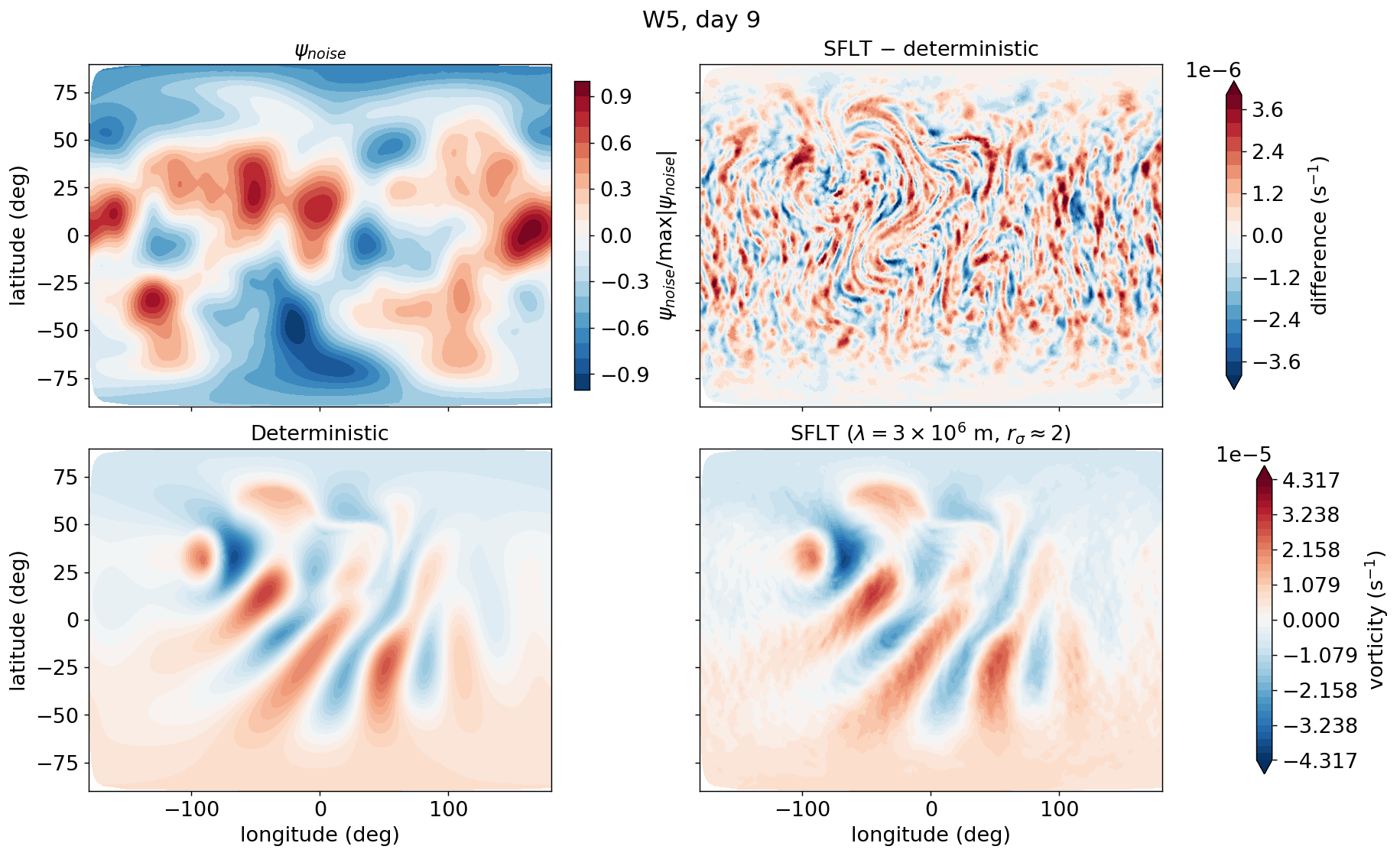}
  \caption{As Figure~\ref{fig:w5-lam3e6-day1}, day~$9$.}
  \label{fig:w5-lam3e6-day9}
\end{figure}

\begin{figure}
  \centering
  \includegraphics[width=0.65\textwidth]{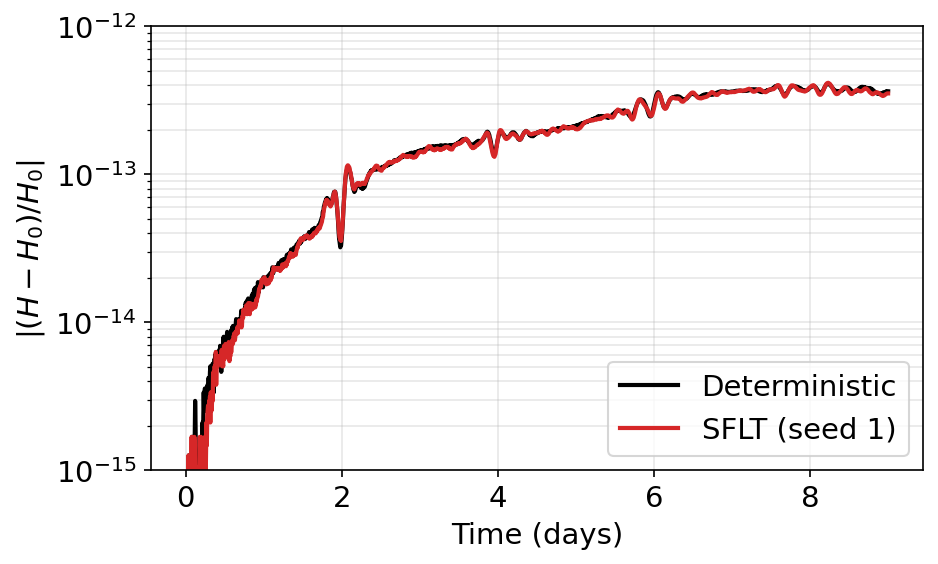}
  \caption{Relative energy error $|(H-H_0)/H_0|$, deterministic and SFLT
    (seed~1), for the run of Figure~\ref{fig:w5-lam3e6-day1}.
    $\Delta t = 432$\,s.}
  \label{fig:w5-r2sflt-energy}
\end{figure}

\begin{figure}
  \centering
  \includegraphics[width=0.65\textwidth]{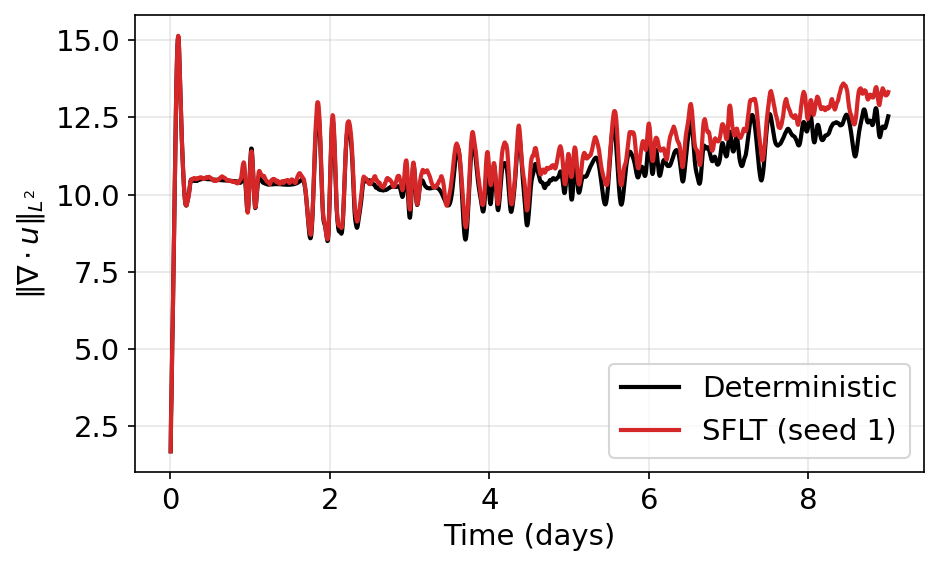}
  \caption{Velocity divergence $\|\nabla\cdot \bm{u}\|_{L^2}$ for the runs of
    Figure~\ref{fig:w5-r2sflt-energy}.}
  \label{fig:w5-r2sflt-div}
\end{figure}

\section{Conclusions and outlook}\label{sec:discussion}
In this paper we combined a compatible finite element discretisation
of the rotating shallow water equations with a continuous
Petrov--Galerkin discretisation in time. The scheme conserves energy
exactly, by the antisymmetry of the discrete bracket, provided the
time integrals are evaluated exactly. We extended this construction to
the rotating shallow water equations with SFLT transport noise, which
enters through the same antisymmetric bracket, so the energy identity
holds pathwise. We provided a robust and scalable iterative solver
strategy based upon a monolithic additive Schwarz method using star
patches. The numerical results confirm the expected convergence
orders and energy conservation up to iterative solver truncation,
independent of the noise amplitude and of the realisation.
These results also give an initial qualitative look at the behaviour of
SFLT applied to the rotating shallow water equations on the sphere.
The noise injects momentum into the solution, leading to local
vortex structures, whilst conserving energy. The purpose of this
noise is to model the effects of small scale unresolved motion on
the resolved dynamics, but we leave the problem of estimating
the statistics of the noise process $\bf{B}$ to future work.

\paragraph{Acknowledgements.} The authors acknowledge funding from EPSRC grant EP/W016125/1 ``Next generation
particle filters for stochastic partial differential equations'', and
the European Research Council Synergy grant ``Stochastic Transport in
Upper Ocean Dynamics''. We are grateful to Darryl Holm for supporting
this work.

\bibliography{epac_stoch}

\begin{thebibliography}{10}

\bibitem{andrews2026irksome}
B.~D. Andrews, P.~Brubeck, P.~E. Farrell, R.~C. Kirby, and S.~P. MacLachlan.
\newblock Automated {G}alerkin time stepping in {Irksome}.
\newblock {\em arXiv preprint arXiv: 2606.27300}, 2026.

\bibitem{andrew_farrell_sisc25}
B.~D. Andrews and P.~E. Farrell.
\newblock Enforcing conservation laws and dissipation inequalities numerically
  via auxiliary variables.
\newblock {\em SIAM J. Sci. Comput.}, 47(6):A3516--A3535, 2025.

\bibitem{arakawa1990energy}
A.~Arakawa and Y.-J.~G. Hsu.
\newblock Energy conserving and potential-enstrophy dissipating schemes for the
  shallow water equations.
\newblock {\em Mon. Weather Rev.}, 118(10):1960--1969, 1990.

\bibitem{arakawa1981potential}
A.~Arakawa and V.~R. Lamb.
\newblock A potential enstrophy and energy conserving scheme for the shallow
  water equations.
\newblock {\em Mon. Weather Rev.}, 109(1):18--36, 1981.

\bibitem{bauer2018energy}
W.~Bauer and C.~J. Cotter.
\newblock Energy--enstrophy conserving compatible finite element schemes for
  the rotating shallow water equations with slip boundary conditions.
\newblock {\em J. Comput. Phys.}, 373:171--187, 2018.

\bibitem{bauer2025fully}
W.~Bauer and C.~J. Cotter.
\newblock Fully implicit timestepping methods for the rotating shallow water
  equations.
\newblock {\em arXiv preprint arXiv:2508.02358}, 2025.

\bibitem{brecht_JAMES}
R.~Brecht, L.~Li, W.~Bauer, and E.~Mémin.
\newblock Rotating shallow water flow under location uncertainty with a
  structure-preserving discretization.
\newblock {\em J. Adv. Model. Earth Syst.}, 13(12):e2021MS002492, 2021.

\bibitem{cohen2011linear}
D.~Cohen and E.~Hairer.
\newblock Linear energy-preserving integrators for poisson systems.
\newblock {\em BIT Numer. Math.}, 51(1):91--101, 2011.

\bibitem{cotter2019numerically}
C.~Cotter, D.~Crisan, D.~D. Holm, W.~Pan, and I.~Shevchenko.
\newblock Numerically modeling stochastic {L}ie transport in fluid dynamics.
\newblock {\em Multiscale Model. Simul.}, 17(1):192--232, 2019.

\bibitem{cotter2023compatible}
C.~J. Cotter.
\newblock Compatible finite element methods for geophysical fluid dynamics.
\newblock {\em Acta Numer.}, 32:291--393, 2023.

\bibitem{CotterShipton2023}
C.~J. Cotter and J.~Shipton.
\newblock A compatible finite element discretisation for the nonhydrostatic
  vertical slice equations.
\newblock {\em GEM Int. J. Geomath.}, 14(1):25, 2023.

\bibitem{cullen2007modelling}
M.~Cullen.
\newblock Modelling atmospheric flows.
\newblock {\em Acta Numer.}, 16:67--154, 2007.

\bibitem{dubos2014equations}
T.~Dubos and M.~Tort.
\newblock Equations of atmospheric motion in non-eulerian vertical coordinates:
  Vector-invariant form and quasi-{H}amiltonian formulation.
\newblock {\em Mon. Weather Rev.}, 142(10):3860--3880, 2014.

\bibitem{egger2021energy}
H.~Egger, O.~Habrich, and V.~Shashkov.
\newblock On the energy stable approximation of {H}amiltonian and gradient
  systems.
\newblock {\em Comput. Methods Appl. Math.}, 21(2):335--349, 2021.

\bibitem{eidnes2022order}
S{\o}lve. Eidnes.
\newblock Order theory for discrete gradient methods.
\newblock {\em BIT Numer. Math.}, 62(4):1207--1255, 2022.

\bibitem{eldred2017total}
C.~Eldred and D.~Randall.
\newblock Total energy and potential enstrophy conserving schemes for the
  shallow water equations using {H}amiltonian methods -- {P}art 1: {D}erivation
  and properties.
\newblock {\em Geosci. Model Dev.}, 10(2):791--810, 2017.

\bibitem{eldred2019quasi}
Christopher Eldred, Thomas Dubos, and Evaggelos Kritsikis.
\newblock A quasi-hamiltonian discretization of the thermal shallow water
  equations.
\newblock {\em Journal of Computational Physics}, 379:1--31, 2019.

\bibitem{ern2021}
A.~Ern and J.-L. Guermond.
\newblock {\em Finite Elements III : First-Order and TimeDependent PDEs}.
\newblock Springer, 2021.

\bibitem{farrell2021irksome}
P.~E. Farrell, R.~C. Kirby, and J.~Marchena-Menendez.
\newblock {I}rksome: Automating {R}unge--{K}utta time-stepping for finite
  element methods.
\newblock {\em ACM Trans. Math. Softw.}, 47(4):1--26, 2021.

\bibitem{Farrell2021}
P.~E. Farrell, M.~G. Knepley, L.~Mitchell, and F.~Wechsung.
\newblock {PCPATCH}: software for the topological construction of multigrid
  relaxation methods.
\newblock {\em ACM Trans. Math. Softw.}, 47(3):1--22, 2021.

\bibitem{gassmann2013global}
A.~Gassmann.
\newblock A global hexagonal {C}-grid non-hydrostatic dynamical core
  ({ICON-IAP}) designed for energetic consistency.
\newblock {\em Q. J. R. Meteorol. Soc.}, 139(670):152--175, 2013.

\bibitem{gassmann2008towards}
A.~Gassmann and H.-J. Herzog.
\newblock Towards a consistent numerical compressible non-hydrostatic model
  using generalized {H}amiltonian tools.
\newblock {\em Q. J. R. Meteorol. Soc.}, 134(635):1597--1613, 2008.

\bibitem{gawlik2021variational}
E.~S. Gawlik and Fran{\c{c}}ois. Gay-Balmaz.
\newblock A variational finite element discretization of compressible flow.
\newblock {\em Found. Comput. Math.}, 21(4):961--1001, 2021.

\bibitem{gawlik2011geometric}
E.~S. Gawlik, P.~Mullen, D.~Pavlov, J.~E. Marsden, and M.~Desbrun.
\newblock Geometric, variational discretization of continuum theories.
\newblock {\em Physica D}, 240(21):1724--1760, 2011.

\bibitem{giesselmann2025energy}
J.~Giesselmann, A.~Karsai, and T.~Tscherpel.
\newblock Energy-consistent {P}etrov--{G}alerkin time discretization of
  port-{H}amiltonian systems.
\newblock {\em SMAI J. Comput. Math.}, 11:335--367, 2025.

\bibitem{gonzalez1996time}
O.~Gonzalez.
\newblock Time integration and discrete {H}amiltonian systems.
\newblock {\em J. Nonlinear Sci.}, 6(5):449--467, 1996.

\bibitem{hairer2010energy}
E.~Hairer.
\newblock Energy-preserving variant of collocation methods.
\newblock {\em JNAIAM}, 5(1-2):73--84, 2010.

\bibitem{hairer1999numerical}
E.~Hairer, C.~Lubich, and G.~Wanner.
\newblock Numerical geometric integration.
\newblock {\em Unpublished lecture notes}, 1999.

\bibitem{hirani}
A.~N. Hirani.
\newblock {\em Discrete Exterior Calculus}.
\newblock PhD thesis, California Institute of Technology, 2003.

\bibitem{holm2015variational}
D.~D. Holm.
\newblock Variational principles for stochastic fluid dynamics.
\newblock {\em Proc. R. Soc. A: Math. Phys. Eng. Sci.}, 471(2176):20140963,
  2015.

\bibitem{holm2026geometry}
D.~D. Holm.
\newblock The geometry of stochastic fluid dynamics.
\newblock {\em arXiv preprint arXiv: 2608.15321}, 2026.

\bibitem{holm2021stochastic}
D.~D. Holm and R.~Hu.
\newblock Stochastic effects of waves on currents in the ocean mixed layer.
\newblock {\em J. Math. Phys.}, 62(7), 2021.

\bibitem{holm1998euler}
D.~D. Holm, J.~E. Marsden, and T.~S. Ratiu.
\newblock The {Euler--P}oincar{\'e} equations and semidirect products with
  applications to continuum theories.
\newblock {\em Adv. Math.}, 137(1):1--81, 1998.

\bibitem{holm1999euler}
D.~D. Holm, J.~E. Marsden, and T.~S. Ratiu.
\newblock The {Euler-P}oincar{\'e} equations in geophysical fluid dynamics.
\newblock {\em arXiv preprint chao-dyn/9903035}, 1999.

\bibitem{holmsinghsteet2026}
D.~D. Holm, M.~K. Singh, and O.~D. Street.
\newblock A comparative numerical study of stochastic {H}amiltonian
  {Camassa–Holm} equations.
\newblock {\em Proc. R. Soc. A: Math. Phys. Eng. Sci.}, 482(2330):20250635,
  2026.

\bibitem{hupatching2023variational}
R.~Hu and S.~Patching.
\newblock Variational stochastic parameterisations and their applications to
  primitive equation models.
\newblock In B.~Chapron, D.~Crisan, D.~Holm, E.~M{\'e}min, and A.~Radomska,
  editors, {\em Stochastic Transport in Upper Ocean Dynamics}, pages 135--158.
  Springer International Publishing, Cham, 2023.

\bibitem{jablonowski2011pros}
C.~Jablonowski and D.~L. Williamson.
\newblock The pros and cons of diffusion, filters and fixers in atmospheric
  general circulation models.
\newblock {\em Numerical Techniques for Global Atmospheric Models}, pages
  381--493, 2011.

\bibitem{kloeden92}
P.E. Kloeden and E.~Platen.
\newblock {\em Numerical Solution of Stochastic Differential Equations}.
\newblock Springer, 1 edition, 1992.

\bibitem{lauter2005unsteady}
M.~L{\"a}uter, D.~Handorf, and K.~Dethloff.
\newblock Unsteady analytical solutions of the spherical shallow water
  equations.
\newblock {\em J. Comput. Phys.}, 210(2):535--553, 2005.

\bibitem{lee2018mixed}
D.~Lee and A.~Palha.
\newblock A mixed mimetic spectral element model of the rotating shallow water
  equations on the cubed sphere.
\newblock {\em J. Comput. Phys.}, 375:240--262, 2018.

\bibitem{lee2020mixed}
D.~Lee and A.~Palha.
\newblock A mixed mimetic spectral element model of the 3{D} compressible euler
  equations on the cubed sphere.
\newblock {\em J. Comput. Phys.}, 401:108993, 2020.

\bibitem{lee2018discrete}
D.~Lee, A.~Palha, and M.~Gerritsma.
\newblock Discrete conservation properties for shallow water flows using mixed
  mimetic spectral elements.
\newblock {\em J. Comput. Phys.}, 357:282--304, 2018.

\bibitem{leimkuhler2004simulating}
B.~Leimkuhler and S.~Reich.
\newblock {\em Simulating {H}amiltonian dynamics}, volume~14.
\newblock Cambridge University Press, 2004.

\bibitem{mclachlan1999geometric}
R.~I. McLachlan, G.~R.~W. Quispel, and N.~Robidoux.
\newblock Geometric integration using discrete gradients.
\newblock {\em Philos. Trans. R. Soc. A}, 357(1754):1021--1045, 1999.

\bibitem{mclaren2003bootstrapping}
D.~I. McLaren and G.~R.~W. Quispel.
\newblock Bootstrapping discrete-gradient integral-preserving integrators to
  fourth order.
\newblock {\em Nonlinear Dyn.}, pages 157--172, 2003.

\bibitem{mclaren2004integral}
D.~I. McLaren and G.~R.~W. Quispel.
\newblock Integral-preserving integrators.
\newblock {\em J. Phys. A}, 37(39):L489--L495, 2004.

\bibitem{mcrae2014energy}
A.~T.~T. McRae and C.~J. Cotter.
\newblock Energy-and enstrophy-conserving schemes for the shallow-water
  equations, based on mimetic finite elements.
\newblock {\em Q. J. R. Meteorol. Soc.}, 140(684):2223--2234, 2014.

\bibitem{memin2014fluid}
E.~M\'emin.
\newblock Fluid flow dynamics under location uncertainty.
\newblock {\em Geophys. Astrophys. Fluid Dyn.}, 108(2):119--146, 2014.

\bibitem{morrison1998}
P.~J. Morrison.
\newblock {H}amiltonian description of the ideal fluid.
\newblock {\em Rev. Mod. Phys.}, 70(2):467, 1998.

\bibitem{morrison1981hamiltonian}
P.~J. Morrison et~al.
\newblock Hamiltonian field description of the one-dimensional {Poisson-Vlasov}
  equations.
\newblock Technical report, Princeton Univ., NJ (USA). Plasma Physics Lab,
  1981.

\bibitem{natale2017scale}
A.~Natale and C.~J. Cotter.
\newblock Scale-selective dissipation in energy-conserving finite-element
  schemes for two-dimensional turbulence.
\newblock {\em Q. J. R. Meteorol. Soc.}, 143(705):1734--1745, 2017.

\bibitem{natale2018variational}
A.~Natale and C.~J. Cotter.
\newblock A variational finite-element discretization approach for perfect
  incompressible fluids.
\newblock {\em IMA J. Numer. Anal.}, 38(3):1388--1419, 2018.

\bibitem{nevir2009energy}
P.~N{\'e}vir and M.~Sommer.
\newblock Energy--vorticity theory of ideal fluid mechanics.
\newblock {\em J. Atmos. Sci.}, 66(7):2073--2084, 2009.

\bibitem{norton2013projection}
R.~A. Norton, D.~I. McLaren, G.~R.~W. Quispel, A.~Stern, and A.~Zanna.
\newblock Projection methods and discrete gradient methods for preserving first
  integrals of odes.
\newblock {\em arXiv preprint arXiv:1302.2713}, 2013.

\bibitem{pavlov2011structure}
D.~Pavlov, P.~Mullen, Y.~Tong, E.~Kanso, J.~E. Marsden, and M.~Desbrun.
\newblock Structure-preserving discretization of incompressible fluids.
\newblock {\em Physica D}, 240(6):443--458, 2011.

\bibitem{quispel2008new}
G.~R.~W. Quispel and D.~I. McLaren.
\newblock A new class of energy-preserving numerical integration methods.
\newblock {\em J. Phys. A}, 41(4):045206, 2008.

\bibitem{Amin2025}
A.~Rafiei and S.~MacLachlan.
\newblock Achieving h‐and p‐robust monolithic multigrid solvers for the
  {S}tokes equations.
\newblock {\em Numer. Linear Algebra Appl.}, 32(3), 2025.

\bibitem{resseguier2017geophysical}
V.~Resseguier, E.~M\'emin, and B.~Chapron.
\newblock Geophysical flows under location uncertainty, {P}art {I}: Random
  transport and general models.
\newblock {\em Geophys. Astrophys. Fluid Dyn.}, 111(3):149--176, 2017.

\bibitem{ringler2010unified}
T.~D. Ringler, J.~Thuburn, J.~B. Klemp, and W.~C. Skamarock.
\newblock A unified approach to energy conservation and potential vorticity
  dynamics for arbitrarily-structured {C}-grids.
\newblock {\em J. Comput. Phys.}, 229(9):3065--3090, 2010.

\bibitem{sadourny1985parameterization}
R.~Sadourny and C.~Basdevant.
\newblock Parameterization of subgrid scale barotropic and baroclinic eddies in
  quasi-geostrophic models: Anticipated potential vorticity method.
\newblock {\em J. Atmos. Sci.}, 42(13):1353--1363, 1985.

\bibitem{salmon2004poisson}
R.~Salmon.
\newblock Poisson-bracket approach to the construction of energy-and
  potential-enstrophy-conserving algorithms for the shallow-water equations.
\newblock {\em J. Atmos. Sci.}, 61(16):2016--2036, 2004.

\bibitem{salmon2005general}
R.~Salmon.
\newblock A general method for conserving quantities related to potential
  vorticity in numerical models.
\newblock {\em Nonlinearity}, 18(5):R1--R16, 2005.

\bibitem{schieweck2010stable}
F.~Schieweck.
\newblock A-stable discontinuous {Galerkin--Petrov} time discretization of
  higher order.
\newblock {\em J. Numer. Math.}, 18(1):25, 2010.

\bibitem{sommer2009conservative}
M.~Sommer and P.~N{\'e}vir.
\newblock A conservative scheme for the shallow-water system on a staggered
  geodesic grid based on a nambu representation.
\newblock {\em Q. J. R. Meteorol. Soc.}, 135(639):485--494, 2009.

\bibitem{thuburn2012framework}
J.~Thuburn and C.~J. Cotter.
\newblock A framework for mimetic discretization of the rotating shallow-water
  equations on arbitrary polygonal grids.
\newblock {\em SIAM J. Sci. Comput.}, 34(3):B203--B225, 2012.

\bibitem{thuburn2014mimetic}
J.~Thuburn, C.~J. Cotter, and T.~Dubos.
\newblock A mimetic, semi-implicit, forward-in-time, finite volume shallow
  water model: comparison of hexagonal--icosahedral and cubed-sphere grids.
\newblock {\em Geosci. Model Dev.}, 7(3):909--929, 2014.

\bibitem{thuburn2014cascades}
J.~Thuburn, J.~Kent, and N.~Wood.
\newblock Cascades, backscatter and conservation in numerical models of
  two-dimensional turbulence.
\newblock {\em Q. J. R. Meteorol. Soc.}, 140(679):626--638, 2014.

\bibitem{thuburn2009numerical}
J.~Thuburn, T.~D. Ringler, W.~C. Skamarock, and J.~B. Klemp.
\newblock Numerical representation of geostrophic modes on arbitrarily
  structured {C}-grids.
\newblock {\em J. Comput. Phys.}, 228(22):8321--8335, 2009.

\bibitem{williamson1992standard}
D.~L. Williamson, J.~B. Drake, J.~J. Hack, R.~Jakob, and P.~N. Swarztrauber.
\newblock A standard test set for numerical approximations to the shallow water
  equations in spherical geometry.
\newblock {\em J. Comput. Phys.}, 102(1):211--224, 1992.

\bibitem{wimmer2020energy}
G.~A. Wimmer, C.~J. Cotter, and W.~Bauer.
\newblock Energy conserving upwinded compatible finite element schemes for the
  rotating shallow water equations.
\newblock {\em J. Comput. Phys.}, 401:109016, 2020.

\bibitem{wimmer2021energy}
G.~A. Wimmer, C.~J. Cotter, and W.~Bauer.
\newblock Energy conserving {SUPG} methods for compatible finite element
  schemes in numerical weather prediction.
\newblock {\em SMAI J. Comput. Math.}, 7:267--300, 2021.

\bibitem{zeitlin1991finite}
V.~Zeitlin.
\newblock Finite-mode analogs of 2{D} ideal hydrodynamics: Coadjoint orbits and
  local canonical structure.
\newblock {\em Physica D}, 49(3):353--362, 1991.

\bibitem{zeitlin2004self}
V.~Zeitlin.
\newblock Self-consistent finite-mode approximations for the hydrodynamics of
  an incompressible fluid on nonrotating and rotating spheres.
\newblock {\em Phys. Rev. Lett.}, 93(26):264501, 2004.

\end{thebibliography}

\appendix
\section{Definition and discretisation of the stochastic vector field}
\label{A:B}
In \cite{holm2021stochastic}, the stochastic vector fields are
constructed using a basis expansion; the idea is that the basis would
be constructed from data.  In this article, our focus is on the energy
conserving discretisation, so to keep things simple we just describe
$\bm{B}$ as a stochastic process in space and time. In particular, 
for our numerical experiments we choose $\diff\bm{B}=\sigma\nabla^\perp \diff\psi$,
for a chosen scaling parameter $\sigma$.
Here $\psi$ is a Gaussian process in space and a Stratonovich
process in time (the latter being a requirement of the SFLT framework),
defined as the solution of 
\begin{equation}
  \left( \mathcal{I} - \kappa^{-2}\Delta \right)^{m} \diff\psi = \,\diff W,
  \qquad \kappa = \frac{\sqrt{8\nu}}{\lambda},
  \label{eq:noise-spde}
\end{equation}
where ${W}$ is a Gaussian process that is white in space and
Stratonovich in time, $\eta$ is a scaling factor, $\nu$ is the
smoothness parameter, $\lambda$ is the correlation length, and $m$ the
number of applications of the inverse Helmholtz operator.
On the sphere, $\Delta$ is the Laplace-Beltrami operator.
In our numerical simulations,  we have taken $m=2$ and $\nu = 1$. 

In our implementation, we discretise the white noise process
increments
\begin{equation}
  \delta W^{n}(x) = \int_{t_n}^{t_{n+1}}\diff W(x,t)
  = \sum_{e=1}^{N_e} \delta W_{n,e}|e|^{1/2}\chi_e(x),
\end{equation}
where $\chi_e$ is the indicator function for cell $e$ for
$e=1,\ldots,N_e$ with areas $|e|$, $N_e$ is the number of cells, and
$\delta W_{n,e}$ are i.i.d. $N(0,\sqrt{\Delta t})$ random
variables. For a continuous finite element space $\Psi$, we discretise
\eqref{eq:noise-spde} by seeking $\psi_1,\dots,\psi_m \in \Psi$ with
\begin{align}
  \left\langle \psi_1, \chi \right\rangle
    + \kappa^{-2}\left\langle \nabla \delta \psi_1, \nabla \chi \right\rangle
    &= \left\langle \delta {W}_h, \chi \right\rangle ,
    &&\forall \chi \in \Psi, \label{eq:noise-solve1}\\
  \left\langle \delta \psi_{j+1}, \chi \right\rangle
    + \kappa^{-2}\left\langle \nabla \delta \psi_{j+1}, \nabla \chi \right\rangle
    &= \left\langle \delta \psi_j, \chi \right\rangle ,
    &&\forall \chi \in \Psi, \quad j = 1,\dots,m-1, \label{eq:noise-solve2}
\end{align}
and finally set $\delta {\bm{B}} = \delta\psi_m$. In our implementation,
we took $\Psi$ to be the continuous quadratic finite element space,
so that $\delta \bm{B} \in V$.

To report the strength of the forcing, we define the noise ratio,
\begin{equation}
  r_\sigma \;:=\; \frac{\big\| \delta{\bm{B}} / D \big\|_{L^2}}
    {\big\| \bm{u} \big\|_{L^2}},
  \label{eq:noise-level}
\end{equation}
evaluated at the initial condition.

\end{document}